\documentclass[10pt,a4paper]{article}

\usepackage{amsfonts,amsmath,amssymb,amsthm,color,graphicx,latexsym,mathrsfs,url}

\newtheorem{theorem}{Theorem}
\newtheorem{lemma}[theorem]{Lemma}

\newtheorem{remark}[theorem]{Remark}
\newtheorem{proposition}[theorem]{Proposition}
\newtheorem{definition}[theorem]{Definition}
\newtheorem{corollary}[theorem]{Corollary}

\usepackage[dvipdfm,
linkbordercolor={1 0 0},
citebordercolor={0 1 0},
urlbordercolor={0 1 1}]{hyperref}

\newcommand{\rd}{\,\mathrm{d}}
\newcommand{\rtr}{\,\mathrm{tr}}

\newcommand{\bsk}{\boldsymbol{k}}
\newcommand{\bsl}{\boldsymbol{l}}
\newcommand{\bsq}{\boldsymbol{q}}

\newcommand{\bsx}{\boldsymbol{x}}
\newcommand{\bsy}{\boldsymbol{y}}
\newcommand{\bsz}{\boldsymbol{z}}

\newcommand{\bsalpha}{\boldsymbol{\alpha}}

\newcommand{\bsgamma}{\boldsymbol{\gamma}}

\newcommand{\bssigma}{\boldsymbol{\sigma}}
\newcommand{\bstau}{\boldsymbol{\tau}}

\newcommand{\bszero}{\boldsymbol{0}}

\newcommand{\nat}{\mathbb{N}}

\newcommand{\FF}{\mathbb{F}}
\newcommand{\RR}{\mathbb{R}}
\newcommand{\ZZ}{\mathbb{Z}}

\newcommand{\Dcal}{\mathcal{D}}
\newcommand{\Ecal}{\mathcal{E}}

\newcommand{\Hcal}{\mathcal{H}}

\newcommand{\Kcal}{\mathcal{K}}
\newcommand{\Lcal}{\mathcal{L}}

\newcommand{\sh}{\mathrm{sh}}
\newcommand{\tr}{\mathrm{tr}}
\newcommand{\wal}{\mathrm{wal}}

\allowdisplaybreaks

\begin{document}

\title{The $b$-adic tent transformation for quasi-Monte Carlo integration using digital nets\thanks{The work of the first author was supported by Grant-in-Aid for JSPS Fellows No.24-4020. The works of the second and third authors were supported by the Program for Leading Graduate Schools, MEXT, Japan.}}

\author{Takashi Goda\thanks{Graduate School of Engineering, The University of Tokyo, 7-3-1 Hongo, Bunkyo-ku, Tokyo 113-8656 ({\tt goda@frcer.t.u-tokyo.ac.jp}).}, Kosuke Suzuki\thanks{Graduate School of Mathematical Sciences, The University of Tokyo, 3-8-1 Komaba, Meguro-ku, Tokyo 153-8914 ({\tt ksuzuki@ms.u-tokyo.ac.jp}).}, Takehito Yoshiki\thanks{Graduate School of Mathematical Sciences, The University of Tokyo, 3-8-1 Komaba, Meguro-ku, Tokyo 153-8914 ({\tt yosiki@ms.u-tokyo.ac.jp}).}}

\date{\today}

\maketitle

\begin{abstract}
In this paper we investigate quasi-Monte Carlo (QMC) integration using digital nets over $\ZZ_b$ in reproducing kernel Hilbert spaces. The tent transformation, or the baker's transformation, was originally used for lattice rules by Hickernell (2002) to achieve higher order convergence of the integration error for smooth non-periodic integrands, and later, has been successfully applied to digital nets over $\ZZ_2$ by Cristea et al. (2007) and Goda (2014). The aim of this paper is to generalize the latter two results to digital nets over $\ZZ_b$ for an arbitrary prime $b$. For this purpose, we introduce the {\em $b$-adic tent transformation} for an arbitrary positive integer $b$ greater than 1, which is a generalization of the original (dyadic) tent transformation. Further, again for an arbitrary positive integer $b$ greater than 1, we analyze the mean square worst-case error of QMC rules using digital nets over $\ZZ_b$ which are randomly digitally shifted and then folded using the $b$-adic tent transformation in reproducing kernel Hilbert spaces. Using this result, for a prime $b$, we prove the existence of good higher order polynomial lattice rules over $\ZZ_b$ among the smaller number of candidates as compared to the result by Dick and Pillichshammer (2007), which achieve almost the optimal convergence rate of the mean square worst-case error in unanchored Sobolev spaces of smoothness of arbitrary high order.
\end{abstract}

\section{Introduction}\label{sec:intro}
In this paper we are interested in approximating multivariate integrals of functions defined over the $s$-dimensional unit cube
  \begin{align*}
     I(f):=\int_{[0,1]^s}f(\bsx)\rd \bsx ,
  \end{align*}
by quasi-Monte Carlo (QMC) rules
  \begin{align*}
     Q(f;P_{N,s}) := \frac{1}{N}\sum_{\bsx\in P_{N,s}}f(\bsx) ,
  \end{align*}
where $P_{N,s}\subset [0,1]^s$ is a point set consisting of $N$ points. In order to obtain a small integration error, we need to choose $P_{N,s}$ carefully depending on the class of integrands under consideration. Two prominent ways to construct good point sets which are known are integration lattices, see, e.g., \cite{N92b,SJ94}, and digital nets and sequences, see, e.g., \cite{DP10,N92b}. QMC rules based on integration lattices are usually called lattice rules. In this paper, we focus on QMC rules using digital nets as point sets.

The typical convergence rate of the integration error using QMC rules is $O(N^{-1+\varepsilon})$ with arbitrarily small $\varepsilon >0$. In order to achieve higher order convergence of the integration error, it is of interest to study how to construct point sets which can exploit the smoothness of an integrand. It has long been known that it is possible to achieve higher order convergence for smooth periodic integrands by using lattice rules, whereas neither lattice rules nor QMC rules using digital nets can exploit the smoothness of non-periodic integrands so as to achieve higher order convergence. More recently, regarding QMC rules using digital nets, Dick \cite{D07,D08} analyzed the decay of Walsh coefficients of smooth periodic and non-periodic functions, respectively, and introduced higher order digital nets that can achieve higher order convergence. Higher order polynomial lattice point sets, which were first studied in \cite{DP07} by generalizing the definition of polynomial lattice point sets in \cite{N92a}, are one of the special examples of higher order digital nets. (In this paper, we shall use the word {\em digital nets} as a generic term that includes higher order digital nets.) Regarding lattice rules, on the other hand, the tent transformation, also known as the baker's transformation, was used by Hickernell \cite{H02} to achieve higher order convergence for non-periodic integrands in unanchored Sobolev spaces of smoothness of second order. Here we note that the tent transformation has been originally introduced and studied in the context of dynamical systems, see, e.g., \cite{R99}.

The tent transformation was later analyzed in the context of QMC rules using digital nets by Cristea et al. \cite{CDLP07}, where the tent transformation was successfully applied to randomly digitally shifted digital nets over $\ZZ_2$ to achieve almost the optimal convergence rate for integrands in unanchored Sobolev spaces of smoothness of second order. Their result has been generalized very recently by one of the authors \cite{Gxx} to unanchored Sobolev spaces of smoothness of arbitrary high order for the purpose of constructing good higher order polynomial lattice rules over $\ZZ_2$ with modulus of reduced degree as compared to \cite{BDGP11,DP07}. We refer to Subsection \ref{subsec:hopoly} for what {\em modulus} means here.

The aim of this paper is to further generalize the last two studies to QMC rules using digital nets over $\ZZ_b$, where $b$ is an arbitrary positive integer greater than 1. Since the original tent transformation no longer meets this purpose, we need to introduce the {\em $b$-adic tent transformation} ($b$-TT), which is a generalization of the original (dyadic) tent transformation. Another generalization of the tent transformation was studied in \cite{BFS14}, but is different from ours. Employing digital nets over $\ZZ_b$ that are randomly digitally shifted and then folded using the $b$-TT as point sets, the mean square worst-case error in reproducing kernel Hilbert spaces can be analyzed in a way analogous to \cite[Section~3]{CDLP07}. Using this result, for a prime $b$, we can prove the existence of good higher order polynomial lattice rules over $\ZZ_b$ with modulus of reduced degree as compared to \cite{BDGP11,DP07}, which achieve almost the optimal convergence rate of the mean square worst-case error in unanchored Sobolev spaces of smoothness of arbitrary high order, as shown in \cite[Section~4]{Gxx} for the case $b=2$. This means that we can find good higher order polynomial lattice rules among the smaller number of candidates. Hence it would be of interest to study how to construct such good rules efficiently in a manner similar to \cite[Section~6]{Gxx}, but answering this question is beyond the scope of this paper and will be discussed in \cite{GSYxx}.

The remainder of this paper is organized as follows. In Section \ref{sec:pre}, we recall the necessary background and notation, including Walsh functions, digital nets and higher order polynomial lattice point sets. In Section \ref{sec:bbt}, we introduce the $b$-TT and describe its properties that will be used in the subsequent analysis. We then investigate the mean square worst-case error of QMC rules using digital nets over $\ZZ_b$ that are randomly digitally shifted and then folded using the $b$-TT in reproducing kernel Hilbert spaces in Section \ref{sec:mse}. Finally, in Section \ref{sec:hopoly}, we consider unanchored Sobolev spaces of smoothness of arbitrary high order, and for a prime $b$, we prove the existence of good higher order polynomial lattice rules over $\ZZ_b$ with modulus of reduced degree, which achieve almost the optimal convergence rate of the mean square worst-case error.

\section{Preliminaries}\label{sec:pre}
Throughout this paper, we shall use the following notation. Let $\nat$ be the set of positive integers and let $\nat_0:=\nat\cup \{0\}$. For a positive integer $b\ge 2$, let $\ZZ_b$ be a finite ring with $b$ elements, which we identify with the set $\{0,1,\dots,b-1\}$ equipped with addition and multiplication modulo $b$. For $x\in [0,1]$, its $b$-adic expansion $x=\sum_{i=1}^{\infty}\xi_i b^{-i}$, with $\xi_i\in \ZZ_b$ for all $i$, is unique in the sense that infinitely many of the $\xi_i$ are different from $b-1$ if $x \neq 1$ and that all $\xi_i$ are equal to $b-1$ if $x = 1$. The operators $\oplus$ and $\ominus$ denote digitwise addition and subtraction modulo $b$, respectively. That is, for $x, x'\in [0,1]$ whose unique $b$-adic expansions are $x=\sum_{i=1}^{\infty}\xi_i b^{-i}$ and $x'=\sum_{i=1}^{\infty}\xi'_i b^{-i}$ with $\xi_i,\xi'_i\in \ZZ_b$ for all $i$, $\oplus$ and $\ominus$ are defined as
  \begin{align*}
    x\oplus x' = \sum_{i=1}^{\infty}\eta_i b^{-i}\quad \text{and}\quad x\ominus x' = \sum_{i=1}^{\infty}\eta'_i b^{-i},
  \end{align*}
where $\eta_i=\xi_i+\xi'_i \pmod b$ and $\eta'_i=\xi_i-\xi'_i \pmod b$, respectively. Similarly, we define digitwise addition and subtraction for non-negative integers based on their $b$-adic expansions. In case of vectors in $[0,1]^s$ or $\nat_0^s$, the operators $\oplus$ and $\ominus$ are applied componentwise.

\subsection{Walsh functions}\label{subsec:walsh}
Walsh functions often play a central role in the analysis of digital nets. We refer to \cite[Appendix~A]{DP10} for background on Walsh functions. We first give the definition for the one-dimensional case.

\begin{definition}\label{def:wal_1}
Let $b\ge 2$ be a positive integer and let $\omega_b:=\exp(2\pi \sqrt{-1}/b)$ be the primitive $b$-th root of unity. We denote the $b$-adic expansion of $k\in \nat_0$ by $k = \kappa_0+\kappa_1b+\dots+\kappa_{a-1}b^{a-1}$ with $\kappa_i\in \ZZ_b$. Then the $k$-th $b$-adic Walsh function ${}_b\wal_k: [0,1]\to \{1,\omega_b,\dots,\omega_b^{b-1}\}$ is defined as
  \begin{align*}
    {}_b\wal_k(x) := \omega_b^{\kappa_0\xi_1+\dots+\kappa_{a-1}\xi_a} ,
  \end{align*}
for $x\in [0,1]$ with its unique $b$-adic expansion $x=\xi_1b^{-1}+\xi_2b^{-2}+\cdots$.
\end{definition}

\begin{remark}
Walsh functions are usually defined on $[0,1)$. In Definition \ref{def:wal_1}, however, they are defined on $[0,1]$, since in the subsequent analysis we shall need to consider the values of Walsh functions for $x=\xi_1b^{-1}+\xi_2b^{-2}+\cdots$ such that all $\xi_i$ are equal to $b-1$. \end{remark}
\noindent
Definition \ref{def:wal_1} can be generalized to the higher-dimensional case.

\begin{definition}\label{def:wal_s}
Let $b\ge 2$ be a positive integer. For a dimension $s\in \nat$, let $\bsx=(x_1,\ldots, x_s)\in [0,1]^s$ and $\bsk=(k_1,\ldots, k_s)\in \nat_0^s$. Then the $\bsk$-th $b$-adic Walsh function ${}_b\wal_{\bsk}: [0,1]^s \to \{1,\omega_b,\ldots, \omega_b^{b-1}\}$ is defined as
  \begin{align*}
    {}_b\wal_{\bsk}(\bsx) := \prod_{j=1}^s {}_b\wal_{k_j}(x_j) .
  \end{align*}
\end{definition}

Since we shall always use Walsh functions in a fixed base $b$, we omit the subscript and simply write $\wal_k$ or $\wal_{\bsk}$ in the remainder of this paper. Following the exposition in \cite[Appendix~A.2]{DP10}, several important properties of Walsh functions and their related fact are summarized below. In the following, we call $x\in [0,1]$ a $b$-adic rational if $x$ is represented by a finite $b$-adic expansion. Here we note that $x=1$ is not a $b$-adic rational from our unique expansion.

\begin{proposition}\label{prop:walsh}
We have the following:
\begin{enumerate}
\item Let $k,l\in \nat_0$ and $x,y\in [0,1]$. If $x\oplus y$ is not a $b$-adic rational, we have
 $$
    \wal_k(x)\wal_l(x)=\wal_{k\oplus l}(x), \quad \wal_k(x)\wal_k(y)=\wal_k(x\oplus y) .
$$
If $x\ominus y$ is not a $b$-adic rational, we have
$$
    \wal_k(x)\overline{\wal_l(x)}=\wal_{k\ominus l}(x), \quad \wal_k(x)\overline{\wal_k(y)}=\wal_k(x\ominus y) .
$$
\item For $k\in \nat_0$, we have
  \begin{align*}
    \int_0^1 \wal_k(x)\rd x = \left\{ \begin{array}{ll}
     1 & \text{if}\quad k=0  ,  \\
     0 & \text{otherwise} .  \\
    \end{array} \right.
  \end{align*}
\item For all $\bsk,\bsl\in \nat_0^s$, we have
  \begin{align*}
    \int_{[0,1]^s} \wal_{\bsk}(\bsx)\overline{\wal_{\bsl}(\bsx)}\rd \bsx = \left\{ \begin{array}{ll}
     1 & \text{if}\quad \bsk=\bsl  ,  \\
     0 & \text{otherwise} .  \\
    \end{array} \right.
  \end{align*}
\item Let $\bssigma\in [0,1]^s$. Then for any $f\in \Lcal_2([0,1]^s)$, we have
  \begin{align*}
    \int_{[0,1]^s} f(\bsx\oplus \bssigma)\rd \bsx = \int_{[0,1]^s} f(\bsx)\rd \bsx.
  \end{align*}
\item The system $\{\wal_{\bsk}: \bsk\in \nat_0^s\}$ is a complete orthonormal system in $\Lcal_2([0,1]^s)$ for any $s\in \nat$.
\end{enumerate}
\end{proposition}

\begin{remark}\label{remark:walsh}
In Item 1 of Proposition \ref{prop:walsh}, we exclude the case that $x\oplus y$ is a $b$-adic rational and the case that $x\ominus y$ is a $b$-adic rational. However, if we fix one variable, this exclusion holds for an at most countably infinite subset of $[0,1]$. For example, for any $y \in [0,1]$, the set $S =\{ x \mid \text{ $x \oplus y$ is a $b$-adic rational} \} \subset [0,1]$ is countable, and thus the Lebesgue measure of $S$ is 0. For this reason, this exclusion does not violate our subsequent analysis.
\end{remark}

\subsection{Digital nets}\label{subsec:digital_net}
In order to consider digital nets over $\ZZ_b$ for an arbitrary positive integer $b\ge 2$, we introduce the definition of digital nets over $\ZZ_b$ as in \cite{DM13}. For $m,n\in \nat$ with $m\le n$, we consider a point set $P_{b^m,s}\subset [0,1]^s$ consisting of $N=b^m$ points. We assume that every coordinate of $\bsx=(x_1,\dots,x_s)\in P_{b^m,s}$ is expressed with $n$-digit precision, which means that $x_j$ is given in the form of $x_j=\xi_{1,j}b^{-1}+\xi_{2,j}b^{-2}+\dots+\xi_{n,j}b^{-n}$ for $\xi_{1,j},\dots,\xi_{n,j}\in \ZZ_b$. Thus every point $\bsx$ can be identified with one element in $\ZZ_b^{s\times n}$, and similarly, $P_{b^m,s}$ can be identified with a subset of $\ZZ_b^{s\times n}$ consisting of $b^m$ elements. We then call $P_{b^m,s}$ {\em a digital net over $\ZZ_b$} when $P_{b^m,s}$ is identified with a {\em subgroup} of $\ZZ_b^{s\times n}$, where the group operation is the componentwise addition modulo $b$.

We now introduce a well-known construction principle of digital nets over $\ZZ_b$, see \cite{LNS96} for details. Let $C_1,\dots,C_s\in \ZZ_b^{n\times m}$ be $n\times m$ matrices over $\ZZ_b$ with $m\le n$. For $0\le l<b^m$, we denote the $b$-adic expansion of $l$ by $l=\iota_0+\iota_1b+\dots+\iota_{m-1}b^{m-1}$ with $\iota_0,\dots,\iota_{m-1}\in \ZZ_b$. For the vector $\vec{l}=(\iota_0,\dots,\iota_{m-1})^{\top}\in \ZZ_b^m$, we consider
  \begin{align*}
    \vec{y}_{l,j}=C_j\vec{l} \in \ZZ_b^n ,
  \end{align*}
for $1\le j\le s$, where $\vec{y}_{l,j}=(y_{1,l,j},\dots,y_{n,l,j})^{\top}$. Then we define
  \begin{align*}
    x_{l,j}=\frac{y_{1,l,j}}{b}+\frac{y_{2,l,j}}{b^2}+\dots+\frac{y_{n,l,j}}{b^n} \in [0,1] ,
  \end{align*}
for $1\le j\le s$. In this way we obtain the $l$-th point $\bsx_l=(x_{l,1},\dots,x_{l,s})$. Then a point set $P_{b^m,s}=\{\bsx_0,\dots,\bsx_{b^m-1}\}\subset [0,1]^s$ becomes a digital net over $\ZZ_b$. The matrices $C_1,\ldots,C_s$ are called the generating matrices of a digital net $P_{b^m,s}$. 

In this construction principle, how to find good generating matrices $C_1,\dots,C_s$ is of major concern. There have been many good explicit constructions of these matrices proposed by Sobol', Faure, Niederreiter, Niederreiter and Xing as well as others, see \cite[Section~8]{DP10} for more information. Higher order polynomial lattice point sets yield another construction of these matrices, which will be introduced in the next subsection.

The {\em dual net} of a digital net, which is defined as follows, plays an important role in the subsequent analysis.
\begin{definition}\label{def:dual_net}
For $m,n\in \nat$ with $m\le n$, let $P_{b^m,s}$ be a digital net over $\ZZ_b$ (with $n$-digit precision). The dual net of $P_{b^m,s}$, denoted by $D^{\perp}(P_{b^m,s})$, is defined as
  \begin{align*}
     D^{\perp}(P_{b^m,s}):=\{\bsk\in \nat_0^s: \vec{k}_1\cdot \vec{x}_1+ \dots + \vec{k}_s\cdot \vec{x}_s= 0 \pmod b\quad \text{for all}\; \bsx\in P_{b^m,s}\} ,
  \end{align*}
where, for $1\le j\le s$, $\vec{k}_j=(\kappa_{0,j},\dots,\kappa_{n-1,j})^{\top}\in \ZZ_b^n$ for $k_j$ with $b$-adic expansion $k_j=\kappa_{0,j}+\kappa_{1,j}b+\cdots$, which is actually a finite expansion, and $\vec{x}_j=(\xi_{1,j},\dots,\xi_{n,j})^{\top}\in \ZZ_b^n$ for $x_j$ with $b$-adic expansion $x_j=\xi_{1,j}b^{-1}+\dots+\xi_{n,j}b^{-n}$.
\end{definition}
\noindent
Since a digital net is identified with a subgroup of $\ZZ_b^{s\times n}$, the next lemma can be established from Definition \ref{def:dual_net}, which connects a digital net with Walsh functions. This is an obvious adaptation of \cite[Lemma~4.75]{DP10} to our context.
\begin{lemma}\label{lem:dual_Walsh}
Let $P_{b^m,s}$ be a digital net over $\ZZ_b$, and let $D^{\perp}(P_{b^m,s})$ be its dual net. Then we have
  \begin{align*}
     \sum_{\bsx\in P_{b^m,s}}\wal_{\bsk}(\bsx) = \left\{ \begin{array}{ll}
     b^m & \text{if} \ \bsk\in D^{\perp}(P_{b^m,s}) ,  \\
     0 & \text{otherwise} .  \\
    \end{array} \right.
  \end{align*}
\end{lemma}

Randomization of point sets is useful to obtain some statistical information on the integration error. Especially for digital nets, randomization algorithms by a random digital shift and Owen's scrambling have been often discussed in the literature, see, e.g., \cite[Chapter~13]{DP10}. In this paper we shall use a random digital shift. Let $\bssigma=(\sigma_1,\dots,\sigma_s)\in [0,1]^s$ be such that $\sigma_1,\dots,\sigma_s$ are independently and uniformly distributed in $[0,1]$. Then a randomly digitally shifted digital net $P_{b^m,s}\oplus \bssigma$ is obtained by
  \begin{align*}
     P_{b^m,s}\oplus \bssigma=\{\bsx\oplus \bssigma: \bsx\in P_{b^m,s}\},
  \end{align*}
where, as in the beginning of this section, all $\oplus$ operations are actually well-defined. In Sections \ref{sec:mse} and \ref{sec:hopoly}, we shall employ randomly digitally shifted digital nets that are folded using the $b$-TT.

\subsection{Higher order polynomial lattice point sets}\label{subsec:hopoly}
Higher order polynomial lattice point sets are digital nets over $\ZZ_b$ whose construction is based on rational functions over finite fields. Polynomial lattice point sets were originally introduced by Niederreiter in \cite{N92a}, and later, the definition has been generalized to introduce higher order polynomial lattice point sets, see, e.g., \cite{DP07,DP10}.

Throughout this subsection, let $b$ be a prime. We denote by $\ZZ_b[x]$ the set of all polynomials over $\ZZ_b$ and by $\ZZ_b((x^{-1}))$ the field of formal Laurent series over $\ZZ_b$. Every element of $\ZZ_b((x^{-1}))$ can be uniquely expressed in the form
  \begin{align*}
    L=\sum_{l=w}^{\infty}t_lx^{-l},
  \end{align*}
for some integer $w$ and $t_l\in \ZZ_b$. For $n\in \nat$, we define the mapping $v_n$ from $\ZZ_b((x^{-1}))$ to the unit interval $[0,1]$ by
  \begin{align*}
    v_n\left( \sum_{l=w}^{\infty}t_l x^{-l}\right) =\sum_{l=\max(1,w)}^{n}t_l b^{-l}.
  \end{align*}
We shall often identify an integer $n=n_0+n_1b+\cdots\in \nat_0$ with a polynomial $n(x)=n_0+n_1x+\cdots \in \ZZ_b[x]$. Then higher order polynomial lattice point sets are constructed as follows.

\begin{definition}\label{def:hopoly}
For $m, n,s \in \nat$ with $m\le n$, let $p \in \ZZ_b[x]$ with $\deg(p)=n$ and let $\bsq=(q_1,\ldots,q_s) \in (\ZZ_b[x])^s$. A higher order polynomial lattice point set $P_{b^m,s}(\bsq,p)$ consists of $b^m$ points that are given by
  \begin{align*}
    \bsx_h &:= \left( v_{n}\left( \frac{h(x)q_1(x)}{p(x)} \right) , \ldots , v_{n}\left( \frac{h(x)q_s(x)}{p(x)} \right) \right) \in [0,1]^s ,
  \end{align*}
for $0\le h <b^m$. A QMC rule using a higher order polynomial lattice point set is called a higher order polynomial lattice rule with a generating vector $\bsq$ and a modulus $p$.
\end{definition}

\begin{remark}\label{rem:hopoly}
From the viewpoint of the construction principle in the preceding subsection, higher order polynomial lattice point sets are understood as follows. Let us consider the expansions
  \begin{align*}
    \frac{q_j(x)}{p(x)}=\sum_{l=w_j}^{\infty}t_{l}^{(j)}x^{-l} \in \ZZ_b((x^{-1})) ,
  \end{align*}
for $1\le j\le s$, where $w_j$ is an integer and all $t_{l}^{(j)}\in \ZZ_b$. Then the $n\times m$ matrix $C_j=(c_{l,r}^{(j)})$ is obtained by
  \begin{align*}
    c_{l,r}^{(j)}=t_{l+r-1}^{(j)} ,
  \end{align*}
for $1\le j\le s$, $1\le l\le n$ and $1\le r\le m$. Then the matrices $C_1,\ldots,C_s$ are used as the generating matrices of a digital net.
\end{remark}
\noindent
Using \cite[Lemma~15.26]{DP10}, the dual net of a higher order polynomial lattice point set $P_{b^m,s}(\bsq,p)$ can be expressed in a different way from Definition \ref{def:dual_net}.
\begin{lemma}\label{lem:hopoly_dual_net}
For $m,n\in \nat$ with $m\le n$, let $P_{b^m,s}(\bsq,p)$ be a higher order polynomial lattice point set. The dual net of $P_{b^m,s}(\bsq,p)$, denoted by $D^{\perp}(\bsq,p)$, is given as
  \begin{align*}
     D^{\perp}(\bsq,p) := \{ & \bsk=(k_1,\dots,k_s)\in \nat_0^s:\\
                                        & \tr_n(k_1)q_1 + \dots + \tr_n(k_s)q_s\equiv a \pmod p\quad \text{with}\quad \deg(a)<n-m \} ,
  \end{align*}
where we define the truncated polynomial $\tr_n(k)$, associated with $k\in \nat_0$ whose $b$-adic expansion is given by $k=\kappa_0+\kappa_1 b+\cdots$, as
  \begin{align*}
     \tr_n(k)(x)=\kappa_0+\kappa_1 x+\dots +\kappa_{n-1} x^{n-1} .
  \end{align*}
\end{lemma}

\section{The $b$-adic tent transformation and its properties}\label{sec:bbt}
In this section, we introduce the $b$-adic tent transformation ($b$-TT) and describe its properties. Before, we recall that the original (dyadic) tent transformation (2-TT) used by Hickernell \cite{H02} is given by $\phi_2(x)=1-|2x-1|$. The two key properties of $\phi_2$, used in the analysis by Cristea et al. \cite{CDLP07}, are essentially that for any $f\in \Lcal_2([0,1])$ the integral of $f\circ \phi_2$ over the interval $[0,1]$ equals that of $f$, and that the $k$-th dyadic Walsh coefficient of $f\circ \phi_2$ becomes 0 if the dyadic sum-of-digits of $k$ is odd. Here the dyadic sum-of-digits is defined as $\delta_2(k)=\kappa_0+\kappa_1+\cdots$ for $k\in \nat_0$ with dyadic expansion $k=\kappa_0+\kappa_12+\cdots$. (We note that Cristea et al. did not explicitly state the latter property in this way.)

We now introduce the $b$-TT, which shall be denoted by $\phi_b$. Let us denote the $b$-adic expansion of $x\in [0,1]$ by
  \begin{align*}
     x=\sum_{i=1}^{\infty}\frac{\xi_i}{b^i} ,
  \end{align*}
where $\xi_i\in \ZZ_b$ for all $i$, which is unique in the sense that infinitely many of the $\xi_i$ are different from $b-1$ if $x \neq 1$ and that all $\xi_i$ are equal to $b-1$ if $x = 1$. Then $\phi_b$ is given by
  \begin{align*}
     \phi_b(x) := \sum_{i=1}^{\infty}\frac{\eta_i}{b^i} \quad \text{with}\quad \eta_i=\xi_{i+1}- \xi_1 \pmod b .
  \end{align*}
For $\bsx = (x_1, \dots, x_s) \in [0,1]^s$, we define $\phi_b(\bsx) := (\phi_b(x_1), \dots, \phi_b(x_s))$ (we use the same symbol $\phi_b$).

In order to give another expression of $\phi_b$, we define two more functions $\sigma_b$ and $\tau_b$ as
  \begin{align*}
     \sigma_b(x) := \sum_{i=1}^{\infty}\frac{\xi_{i+1}}{b^i}=bx-\xi_1 ,
  \end{align*}
and
  \begin{align*}
     \tau_b(x) := \sum_{i=1}^{\infty}\frac{\xi_1}{b^i}=\frac{\xi_1}{b-1} ,
  \end{align*}
respectively. Here we note that only for the case $\xi_1=b-1$ we allow $\tau_b$ to have the $b$-adic expansion with infinitely many digits equal to $b-1$. Using $\sigma_b$ and $\tau_b$, we can express $\phi_b$ as
  \begin{align*}
     \phi_b(x) = \sigma_b(x)\ominus \tau_b(x) = (bx-\xi_1)\ominus \left( \frac{\xi_1}{b-1}\right).
  \end{align*}

\begin{remark}
The $b$-TT is plotted in Figure \ref{fig:tent} for the case $b=3$. For visualization, we consider the {\em truncated} $b$-TT with $n$-digit. That is, we consider the mapping
  \begin{align*}
     \phi_{b,n}(x)=\sum_{i=1}^{n}\frac{\eta_i}{b^i} \quad \text{with}\quad \eta_i=\xi_{i+1}- \xi_1 \pmod b .
  \end{align*}
As can be seen, the $b$-TT is not generally a continuous mapping except for the case $b=2$ or for the interval $[0,1/b)$ in which we have $\phi_b(x)=bx$. For the case $b=2$, we have $\xi_1=1$ in the interval $[1/2,1]$, so that
  \begin{align*}
     \phi_2(x) = & \sum_{i=1}^{\infty}\frac{1-\xi_{i+1}}{2^i} = 1-\sum_{i=1}^{\infty}\frac{\xi_{i+1}}{2^i} \\
     = & 1-(2x-1) = 2-2x ,
  \end{align*}
for $x \in [1/2, 1]$. Hence we recover the original tent transformation as in \cite{H02}.
\end{remark}

\begin{figure}
\begin{center}
\includegraphics{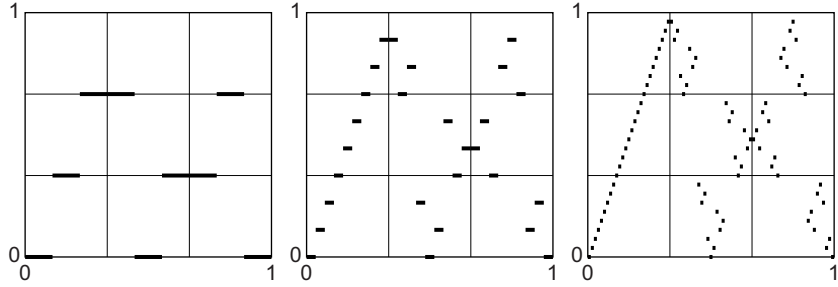}
\caption{The $b$-adic tent transformation with $b=3$ for $n=1,2,3$.}
\label{fig:tent}
\end{center}
\end{figure}
\noindent
In the following, we prove that for any $f\in \Lcal_2([0,1])$ the integral of $f\circ \phi_b$ over the interval $[0,1]$ equals that of $f$, see Theorem \ref{thm:integral_bbt}, and that the $k$-th $b$-adic Walsh coefficient of $f\circ \phi_b$ becomes 0 if the $b$-adic sum-of-digits of $k$ modulo $b$ does not equal 0, see Theorem \ref{thm:Walsh_coeff_bbt}. Here the $b$-adic sum-of-digits of $k$ is defined as $$\delta_b(k):=\kappa_0+\kappa_1+\cdots+\kappa_{a-1}$$ for $k\in \nat_0$ with $b$-adic expansion $k=\kappa_0+\kappa_1b+\cdots+\kappa_{a-1}b^{a-1}$.

\begin{theorem}\label{thm:integral_bbt}
Let $b\ge 2$ be an integer. For any $f\in \Lcal_2([0,1])$, we have
  \begin{align*}
     \int_0^1 (f\circ \phi_b) (x)\rd x = \int_0^1 f(x)\rd x.
  \end{align*}
\end{theorem}
\begin{proof}
By first dividing the interval $[0,1]$ into the $b$ intervals with the same length $[0,1/b),[1/b,2/b),\dots,[(b-1)/b,1]$ and then transforming the variable in each interval, we have
  \begin{align*}
     \int_0^1 (f\circ \phi_b) (x)\rd x = & \sum_{\xi_1=0}^{b-1}\int_{\xi_1/b}^{(\xi_1+1)/b} (f\circ \phi_b)(x)\rd x \\
     = & \sum_{\xi_1=0}^{b-1}\int_{\xi_1/b}^{(\xi_1+1)/b} f\left((bx-\xi_1)\ominus \left(\frac{\xi_1}{b-1} \right)  \right)\rd x \\
     = & \frac{1}{b}\sum_{\xi_1=0}^{b-1}\int_0^1 f\left(y\ominus  \left(\frac{\xi_1}{b-1} \right)  \right)\rd y \\
     = & \frac{1}{b}\sum_{\xi_1=0}^{b-1}\int_0^1 f(y)\rd y = \int_0^1 f(y)\rd y ,
  \end{align*}
where we use Item 4 of Proposition \ref{prop:walsh} in the fourth equality. Hence the result follows.
\end{proof}
\noindent
Since the system $\{\wal_k \colon k\in \nat_0\}$ is a complete orthonormal system in $\Lcal_2([0,1])$ as stated in Item 5 of Proposition \ref{prop:walsh}, we have a Walsh series expansion for any $f\in \Lcal_2([0,1])$,
  \begin{align*}
     f(x) \sim \sum_{k=0}^{\infty}\hat{f}(k)\wal_k(x) ,
  \end{align*}
where the $k$-th Walsh coefficient is given by
  \begin{align*}
     \hat{f}(k) := \int_0^1 f(x)\overline{\wal_k(x)}\rd x .
  \end{align*}
Since we only discuss the Walsh coefficient here, we do not need to have a pointwise absolute convergence of the Walsh series expansion at this moment. In the following theorem, we consider a Walsh series expansion of $f\circ \phi_b$ and calculate the $k$-th Walsh coefficient $\widehat{(f\circ \phi_b)}(k)$. For $k\in \nat_0$ with $b$-adic expansion $k=\kappa_0+\kappa_1 b+\kappa_2 b^2+\cdots+\kappa_{a-1}b^{a-1}$, we define $$\lfloor k/b\rfloor:=\kappa_1 +\kappa_2 b+\cdots+\kappa_{a-1}b^{a-2}.$$

\begin{theorem}\label{thm:Walsh_coeff_bbt}
For any $f\in \Lcal_2([0,1])$ and $k\in \nat_0$, we have 
  \begin{align*}
     \widehat{(f\circ \phi_b)}(k) = \left\{ \begin{array}{ll}
     \hat{f}(\lfloor k/b\rfloor) & \text{if} \quad \delta_b(k)\equiv 0 \pmod b ,  \\
     0 & \text{otherwise} ,  \\
    \end{array} \right.
  \end{align*}
where $\delta_b(k)$ is the $b$-adic sum-of-digits of $k$ as defined above.
\end{theorem}
\begin{proof}
Following the proof of Theorem \ref{thm:integral_bbt}, we have
  \begin{align}\label{eq:Walsh_coeff_bbt0}
     \widehat{(f\circ \phi_b)}(k) = & \int_0^1 (f\circ \phi_b)(x)\overline{\wal_k(x)} \rd x \nonumber \\
     = & \sum_{\xi_1=0}^{b-1}\int_{\xi_1/b}^{(\xi_1+1)/b} (f\circ \phi_b)(x)\overline{\wal_k(x)} \rd x \nonumber \\
     = & \sum_{\xi_1=0}^{b-1}\int_{\xi_1/b}^{(\xi_1+1)/b} f\left((bx-\xi_1)\ominus \left(\frac{\xi_1}{b-1} \right)  \right)\overline{\wal_k(x)} \rd x \nonumber \\
     = & \frac{1}{b}\sum_{\xi_1=0}^{b-1}\int_0^1 f\left(y\ominus \left(\frac{\xi_1}{b-1} \right) \right)\overline{\wal_k\left(\frac{y+\xi_1}{b} \right)} \rd y .
  \end{align}
For any $y\in [0,1)$, we have $(y+\xi_1)/b=(y/b)\oplus (\xi_1/b)$. Using Item 1 of Proposition \ref{prop:walsh} and from the definition of Walsh functions, we have
  \begin{align*}
     \wal_k\left(\frac{y+\xi_1}{b} \right)= \wal_k\left(\frac{\xi_1}{b} \right)\wal_k\left(\frac{y}{b} \right)=\wal_k\left(\frac{\xi_1}{b} \right)\wal_{\lfloor k/b\rfloor}(y) ,
  \end{align*}
for $y \in [0,1)$ except the countably infinite set of points, see Remark \ref{remark:walsh}. Substituting this result to (\ref{eq:Walsh_coeff_bbt0}) and then using Item 4 of Proposition \ref{prop:walsh}, we have
  \begin{align}\label{eq:Walsh_coeff_bbt1}
     \widehat{(f\circ \phi_b)}(k) = & \frac{1}{b}\sum_{\xi_1=0}^{b-1}\overline{\wal_k\left(\frac{\xi_1}{b} \right)}\int_0^1 f\left(y\ominus \left(\frac{\xi_1}{b-1} \right) \right)\overline{\wal_{\lfloor k/b\rfloor}(y)} \rd y \nonumber \\
     = & \frac{1}{b}\sum_{\xi_1=0}^{b-1}\overline{\wal_k\left(\frac{\xi_1}{b} \right)}\int_0^1 f(y)\overline{\wal_{\lfloor k/b\rfloor}\left(y\oplus \left(\frac{\xi_1}{b-1}\right)\right)} \rd y.
  \end{align}
\noindent
In the following, we denote the $b$-adic expansion of $k$ by $k=\kappa_0+\kappa_1 b+\cdots+\kappa_{a-1}b^{a-1}$ with $\kappa_{a-1}\ne 0$. Then the $b$-adic expansion of $\lfloor k/b\rfloor$ can be denoted by $\lfloor k/b\rfloor=\kappa_1+\kappa_2 b+\cdots+\kappa_{a-1}b^{a-2}$. For $\xi_1=0,1,\dots,b-1$, by using Item 1 of Proposition \ref{prop:walsh} and from the definition of Walsh functions, we have
  \begin{align}\label{eq:Walsh_coeff_bbt2}
     \int_0^1 f(y)\overline{\wal_{\lfloor k/b\rfloor}\left(y\oplus \left(\frac{\xi_1}{b-1}\right)\right)} \rd y = & \overline{\wal_{\lfloor k/b\rfloor}\left(\frac{\xi_1}{b-1}\right)}\int_0^1 f(y)\overline{\wal_{\lfloor k/b\rfloor}(y)} \rd y \nonumber \\
     = & \overline{\omega_b^{\kappa_1 \xi_1+\dots + \kappa_{a-1} \xi_1}}\hat{f}(\lfloor k/b\rfloor) \nonumber \\
     = & \omega_b^{(b-\xi_1)\delta_b(\lfloor k/b\rfloor)}\hat{f}(\lfloor k/b\rfloor) .
  \end{align}
Substituting (\ref{eq:Walsh_coeff_bbt2}) into (\ref{eq:Walsh_coeff_bbt1}), we have
  \begin{align*}
     \widehat{(f\circ \phi_b)}(k) = & \frac{1}{b}\sum_{\xi_1=0}^{b-1}\overline{\wal_k\left(\frac{\xi_1}{b} \right)}\omega_b^{(b-\xi_1)\delta_b(\lfloor k/b\rfloor)}\hat{f}(\lfloor k/b\rfloor) \\
     = & \frac{1}{b}\sum_{\xi_1=0}^{b-1}\omega_b^{(b-\xi_1)\kappa_0}\omega_b^{(b-\xi_1)\delta_b(\lfloor k/b\rfloor)} \hat{f}(\lfloor k/b\rfloor) \\
     = & \frac{1}{b}\sum_{\xi_1=0}^{b-1}\omega_b^{(b-\xi_1)\delta_b(k)} \hat{f}(\lfloor k/b\rfloor) \\
     = & \left\{ \begin{array}{ll}
     \hat{f}(\lfloor k/b\rfloor) & \text{if}\quad \delta_b(k) \equiv 0 \pmod b ,  \\
     0 & \text{otherwise} ,  \\
    \end{array} \right.
  \end{align*}
where the third equality stems from the identity $\delta_b(k)=\kappa_0+\delta_b(\lfloor k/b\rfloor)$. Hence, the result follows.
\end{proof}

\section{Mean square worst-case error in reproducing kernel Hilbert spaces}\label{sec:mse}
In this section, we study the mean square worst-case error of QMC rules using digital nets over $\ZZ_b$ that are randomly digitally shifted and then folded using the $b$-TT in reproducing kernel Hilbert spaces. Let us consider a reproducing kernel Hilbert space $\Hcal$ with reproducing kernel $\Kcal:[0,1]^s\times [0,1]^s\to \RR$. The inner product in $\Hcal$ is denoted by $\langle f,g \rangle_{\Hcal}$ for $f,g\in \Hcal$ and the associated norm is denoted by $\| f\|_{\Hcal}:=\sqrt{\langle f,f \rangle_{\Hcal}}$.

It is known that if the reproducing kernel $\Kcal$ satisfies $\int_{[0,1]^s}\sqrt{\Kcal(\bsx,\bsx)}\rd \bsx<\infty$, then the square worst-case error in the space $\Hcal$ with reproducing kernel $\Kcal$ of a QMC rule using a point set $P_{N,s}$ is given by
  \begin{align}\label{eq:worst-case_error}
     e^2(P_{N,s},\Kcal) := & \left( \sup_{\substack{f\in \Hcal \\ \|f\|_{\Hcal}\le 1}}|I(f)-Q(f;P_{N,s})|\right)^2 \nonumber \\ 
     = & \int_{[0,1]^{2s}}\Kcal(\bsx,\bsy)\rd \bsx \rd \bsy-\frac{2}{N}\sum_{\bsx\in P_{N,s}}\int_{[0,1]^s}\Kcal(\bsx,\bsy)\rd \bsy+\frac{1}{N^2}\sum_{\bsx,\bsy\in P_{N,s}}\Kcal(\bsx,\bsy) ,
  \end{align}
and the square initial error is given by
  \begin{align*}
     e^2(P_{0,s},\Kcal) := & \left( \sup_{\substack{f\in \Hcal\\ \|f\|_{\Hcal}\le 1}}|I(f)|\right)^2 = \int_{[0,1]^{2s}}\Kcal(\bsx,\bsy)\rd \bsx \rd \bsy .
  \end{align*}
We refer to \cite[Chapter~2]{DP10} for details. In the following, we always assume $\int_{[0,1]^s}\sqrt{\Kcal(\bsx,\bsx)}\rd \bsx<\infty$ and consider the mean square worst-case error of QMC rules using digital nets over $\ZZ_b$ which are randomly digitally shifted and then folded using the $b$-TT in the space $\Hcal$. For a randomly chosen $\bssigma\in [0,1]^s$ we denote by $\phi_b(P_{N,s}\oplus \bssigma)$ a point set $P_{N,s}$ that is digitally shifted by $\bssigma$ and then folded using the $b$-TT, that is,
  \begin{align*}
      \phi_b(P_{N,s}\oplus \bssigma) := \{\phi_b(\bsx) : \bsx\in P_{N,s}\oplus \bssigma\}.
  \end{align*}
Then the mean square worst-case error $\hat{e}^2(P_{N,s},\Kcal)$ of $\phi_b(P_{N,s}\oplus \bssigma)$ with respect to $\bssigma$ is defined by
  \begin{align*}
     \hat{e}^2(P_{N,s},\Kcal) := & \int_{[0,1]^s}e^2(\phi_b(P_{N,s}\oplus \bssigma),\Kcal) \rd \bssigma .
  \end{align*}
Furthermore the folded digitally shifted reproducing kernel $\Kcal_{\sh,\phi_b}:[0,1]^s\times [0,1]^s \to \RR$ is defined by
  \begin{align*}
     \Kcal_{\sh,\phi_b}(\bsx,\bsy) := \int_{[0,1]^s}\Kcal(\phi_b(\bsx\oplus \bssigma),\phi_b(\bsy\oplus \bssigma))\rd \bssigma .
  \end{align*}
Then we have the following theorem on the mean square worst-case error. Since the result follows in exactly the same way as \cite[Theorem~1]{CDLP07} which requires only the property in our Theorem \ref{thm:integral_bbt}, we omit the proof.
\begin{theorem}\label{thm:mse_bbt}
Let $\Kcal, \Kcal_{\sh,\phi_b}\in \Lcal_2([0,1]^{2s})$ be a reproducing kernel and its folded digitally shifted reproducing kernel, respectively, such that $\int_{[0,1]^s}\sqrt{\Kcal(\bsx,\bsx)}\rd \bsx<\infty$. For a point set $P_{N,s}$, we denote by $\phi_b(P_{N,s}\oplus \bssigma)$ a point set $P_{N,s}$ that is digitally shifted by a randomly chosen $\bssigma$ and then folded using the $b$-TT. Then we have
  \begin{align*}
     \hat{e}^2(P_{N,s},\Kcal) = e^2(P_{N,s},\Kcal_{\sh,\phi_b}) .
  \end{align*}
\end{theorem}

We now consider the Walsh series expansion of $\Kcal_{\sh,\phi_b}$,
  \begin{align*}
     \Kcal_{\sh,\phi_b}(\bsx,\bsy) \sim \sum_{\bsk,\bsl\in \nat_0^s}\hat{\Kcal}_{\sh,\phi_b}(\bsk,\bsl)\wal_{\bsk}(\bsx)\overline{\wal_{\bsl}(\bsy)} ,
  \end{align*}
where the $(\bsk,\bsl)$-th Walsh coefficient is given by
  \begin{align*}
     \hat{\Kcal}_{\sh,\phi_b}(\bsk,\bsl)= \int_{[0,1]^{2s}}\Kcal_{\sh,\phi_b}(\bsx,\bsy)\overline{\wal_{\bsk}(\bsx)}\wal_{\bsl}(\bsy)\rd \bsx \rd \bsy .
  \end{align*}
We shall discuss a pointwise absolute convergence of the Walsh series expansion later in Proposition \ref{prop:the Walsh expansion of kernel}. Regarding the Walsh coefficient of $\Kcal_{\sh,\phi_b}$, we have the following theorem.
\begin{theorem}\label{thm:reproducing_kernel}
Let $\Kcal, \Kcal_{\sh,\phi_b}\in \Lcal_2([0,1]^{2s})$ be a reproducing kernel and its folded digitally shifted reproducing kernel, respectively, such that $\int_{[0,1]^s}\sqrt{\Kcal(\bsx,\bsx)}\rd \bsx<\infty$. For $\bsk,\bsl\in \nat_0^s$, the $(\bsk,\bsl)$-th Walsh coefficient of $\Kcal_{\sh,\phi_b}$ is given by
  \begin{align*}
     \hat{\Kcal}_{\sh,\phi_b}(\bsk,\bsl) = \left\{ \begin{array}{ll}
     \hat{\Kcal}(\lfloor \bsk/b\rfloor,\lfloor \bsk/b\rfloor) & \text{if}\quad \bsk=\bsl\quad \text{and}\quad \bsk\in (\Ecal_b\cup\{0\})^s,  \\
     0 & \text{otherwise} ,  \\
    \end{array} \right.
  \end{align*}
where $\lfloor \bsk/b\rfloor=(\lfloor k_1/b\rfloor,\dots,\lfloor k_s/b\rfloor)$ and we define
  \begin{align*}
     \Ecal_b:=\{k\in \nat: \delta_b(k)\equiv 0\pmod b\}.
  \end{align*}
Moreover, the Walsh series expansion of $\Kcal_{\sh,\phi_b}$ is given by
  \begin{align*}
     \Kcal_{\sh,\phi_b}(\bsx,\bsy) \sim \sum_{\bsk\in (\Ecal_b\cup\{0\})^s}\hat{\Kcal}(\lfloor \bsk/b\rfloor,\lfloor \bsk/b\rfloor)\wal_{\bsk}(\bsx) \overline{\wal_{\bsk}(\bsy)}.
  \end{align*}
\end{theorem}

\begin{proof}
The proof is essentially the same as the proof of \cite[Theorem~2]{CDLP07}. From the definition of $\Kcal_{\sh,\phi_b}$ and applying Item 4 of Proposition \ref{prop:walsh}, we have
  \begin{align*}
     \hat{\Kcal}_{\sh,\phi_b}(\bsk,\bsl) = & \int_{[0,1]^{2s}}\int_{[0,1]^s}\Kcal(\phi_b(\bsx\oplus \bssigma),\phi_b(\bsy\oplus \bssigma)) \overline{\wal_{\bsk}(\bsx)}\wal_{\bsl}(\bsy)\rd \bssigma \rd \bsx \rd \bsy \\
     = & \int_{[0,1]^{2s}}\int_{[0,1]^s}\Kcal(\phi_b(\bsx),\phi_b(\bsy)) \overline{\wal_{\bsk}(\bsx\ominus \bssigma)}\wal_{\bsl}(\bsy\ominus \bssigma)\rd \bssigma \rd \bsx \rd \bsy \\
     = & \int_{[0,1]^{2s}}\Kcal(\phi_b(\bsx),\phi_b(\bsy)) \overline{\wal_{\bsk}(\bsx)}\wal_{\bsl}(\bsy) \rd \bsx \rd \bsy\int_{[0,1]^s}\wal_{\bsk\ominus \bsl}(\bssigma) \rd \bssigma .
  \end{align*}
Here we have from Item 2 of Proposition \ref{prop:walsh} that
  \begin{align*}
     \int_{[0,1]^s}\wal_{\bsk\ominus \bsl}(\bssigma) \rd \bssigma = \left\{ \begin{array}{ll}
     1 & \text{if} \quad \bsk=\bsl ,  \\
     0 & \text{otherwise} .  \\
    \end{array} \right.
  \end{align*}
Thus we have $\hat{\Kcal}_{\sh,\phi_b}(\bsk,\bsl)=0$ for $\bsk\ne \bsl$. Otherwise for $\bsk=\bsl$, by following an argument similar to the proof of Theorem \ref{thm:Walsh_coeff_bbt}, we divide the interval $[0,1]$ into the $b$ intervals with the same length $[0,1/b),[1/b,2/b),\dots,[(b-1)/b,1]$ for $2s$ variables $x_1,\dots,x_s$ and $y_1,\dots,y_s$, and transform the variables in each interval. Then we have
  \begin{align*}
     \hat{\Kcal}_{\sh,\phi_b}(\bsk,\bsk) = & \int_{[0,1]^{2s}}\Kcal(\phi_b(\bsx),\phi_b(\bsy)) \overline{\wal_{\bsk}(\bsx)}\wal_{\bsk}(\bsy) \rd \bsx \rd \bsy \\
     = & \frac{1}{b^{2s}}\sum_{\xi_1,\dots,\xi_s,\eta_1,\dots,\eta_s=0}^{b-1}\prod_{j=1}^{s}\omega_b^{(b-\xi_j)\delta_b(k_j)}\overline{\omega_b^{(b-\eta_j)\delta_b(k_j)}}\hat{\Kcal}(\lfloor \bsk/b\rfloor, \lfloor \bsk/b\rfloor) \\
     = & \prod_{j=1}^{s}\frac{1}{b}\sum_{\xi_j=0}^{b-1}\omega_b^{(b-\xi_j)\delta_b(k_j)}\prod_{j=1}^{s}\frac{1}{b}\sum_{\eta_j=0}^{b-1}\overline{\omega_b^{(b-\eta_j)\delta_b(k_j)}}\hat{\Kcal}(\lfloor \bsk/b\rfloor, \lfloor \bsk/b\rfloor) \\
     = & \left\{ \begin{array}{ll}
     \hat{\Kcal}(\lfloor \bsk/b\rfloor, \lfloor \bsk/b\rfloor) & \text{if}\quad \delta_b(k_j)\equiv 0 \pmod b\quad \text{for}\quad 1\le j\le s ,  \\
     0 & \text{otherwise} .  \\
    \end{array} \right.
  \end{align*}
Hence the result follows.
\end{proof}
\noindent
Combining (\ref{eq:worst-case_error}), and Theorems \ref{thm:mse_bbt} and \ref{thm:reproducing_kernel}, we have a formula for the mean square worst-case error.
\begin{theorem}\label{thm:mse_formula}
Let $\Kcal, \Kcal_{\sh,\phi_b}\in \Lcal_2([0,1]^{2s})$ be a reproducing kernel and its folded digitally shifted reproducing kernel, respectively, such that $\int_{[0,1]^s}\sqrt{\Kcal(\bsx,\bsx)}\rd \bsx<\infty$. Suppose that the Walsh series expansion of $\Kcal_{\sh,\phi_b}$ converges to $\Kcal_{\sh,\phi_b}$ pointwise absolutely. For a point set $P_{N,s}$, we denote by $\phi_b(P_{N,s}\oplus \bssigma)$ the point set $P_{N,s}$ that is digitally shifted by a randomly chosen $\bssigma$ and then folded using the $b$-TT. Then the mean square worst-case error of $\phi_b(P_{N,s}\oplus \bssigma)$ is given by
  \begin{align*}
     \hat{e}^2(P_{N,s},\Kcal) = \sum_{\bsk\in (\Ecal_b\cup\{0\})^s\setminus \{\bszero\}}\hat{\Kcal}(\lfloor \bsk/b\rfloor,\lfloor \bsk/b\rfloor)\frac{1}{N^2}\sum_{\bsx,\bsy\in P_{N,s}}\wal_{\bsk}(\bsx) \overline{\wal_{\bsk}(\bsy)} ,
  \end{align*}
where $\bszero$ is the vector consisting of $s$ zeros. In particular when $P_{b^m,s}$ is a digital net over $\ZZ_b$ with $N=b^m$, we have
  \begin{align*}
     \hat{e}^2(P_{b^m,s},\Kcal) = \sum_{\substack{\bsk\in (\Ecal_b\cup\{0\})^s\setminus \{\bszero\}\\ \bsk\in D^{\perp}(P_{b^m,s})}}\hat{\Kcal}(\lfloor \bsk/b\rfloor,\lfloor \bsk/b\rfloor) ,
  \end{align*}
where $D^{\perp}(P_{b^m,s})$ is the dual net of $P_{b^m,s}$.
\end{theorem}

\begin{proof}
Using Item 4 of Proposition \ref{prop:walsh}, Theorem \ref{thm:integral_bbt} and Fubini's theorem, we have
  \begin{align*}
     \int_{[0,1]^s}\Kcal_{\sh,\phi_b}(\bsx,\bsy)\rd \bsy = & \int_{[0,1]^{2s}}\Kcal(\phi_b(\bsx\oplus \bssigma),\phi_b(\bsy\oplus \bssigma)) \rd \bsy \rd \bssigma \\
     = & \int_{[0,1]^{2s}}\Kcal(\phi_b(\bsx\oplus \bssigma),\phi_b(\bsy)) \rd \bsy \rd \bssigma \\
     = & \int_{[0,1]^{2s}}\Kcal(\phi_b(\bsx\oplus \bssigma),\bsy) \rd \bsy \rd \bssigma \\
     = & \int_{[0,1]^{2s}}\Kcal(\phi_b(\bssigma),\bsy) \rd \bssigma \rd \bsy \\
     = & \int_{[0,1]^{2s}}\Kcal(\bssigma,\bsy) \rd \bssigma \rd \bsy = \hat{\Kcal}(\bszero, \bszero) ,
  \end{align*}
for any $\bsx\in [0,1]^s$. Using this result, it follows from (\ref{eq:worst-case_error}), and Theorems \ref{thm:mse_bbt} and \ref{thm:reproducing_kernel} that for a point set $P_{N,s}$ we have
  \begin{align*}
     \hat{e}^2(P_{N,s},\Kcal) = & -\hat{\Kcal}(\bszero, \bszero)+\frac{1}{N^2}\sum_{\bsx,\bsy\in P_{N,s}}\Kcal_{\sh,\phi_b}(\bsx,\bsy) \\
     = & -\hat{\Kcal}(\bszero, \bszero)+\sum_{\bsk\in (\Ecal_b\cup\{0\})^s}\hat{\Kcal}(\lfloor \bsk/b\rfloor,\lfloor \bsk/b\rfloor)\frac{1}{N^2}\sum_{\bsx,\bsy\in P_{N,s}}\wal_{\bsk}(\bsx)\overline{\wal_{\bsk}(\bsy)} .
  \end{align*}
Hence the first part of the theorem follows.

We now suppose that $P_{b^m,s}$ is a digital net over $\ZZ_b$. As mentioned in Subsection \ref{subsec:digital_net}, $P_{b^m,s}$ can be identified with a subgroup of $\ZZ_b^{s\times n}$ with $n\ge m$. Thus, using Lemma \ref{lem:dual_Walsh}, we have
  \begin{align*}
     \frac{1}{b^{2m}}\sum_{\bsx,\bsy\in P_{b^m,s}}\wal_{\bsk}(\bsx)\overline{\wal_{\bsk}(\bsy)} = & \frac{1}{b^{2m}}\sum_{\bsx\in P_{b^m,s}}\wal_{\bsk}(\bsx) \overline{\sum_{\bsy\in P_{b^m,s}}\wal_{\bsk}(\bsy)}\\
     = & \left\{ \begin{array}{ll}
     1 & \text{if} \quad \bsk\in D^{\perp}(P_{b^m,s}) ,  \\
     0 & \text{otherwise} .  \\
    \end{array} \right.
  \end{align*}
Hence the second part of the theorem follows.
\end{proof}

In Proposition~\ref{prop:the Walsh expansion of kernel} below, we give a sufficient condition that the Walsh series expansion of $\Kcal_{\sh,\phi_b}$ converges to $\Kcal_{\sh,\phi_b}$ pointwise absolutely. For $\bsx = (x_1, \dots, x_s), \bsx' = (x'_1, \dots, x'_s) \in [0,1]^s$
with $b$-adic expansions $x_j = \sum_{i=1}^{\infty}\xi_{i,j} b^{-i}$ and $x'_j = \sum_{i=1}^{\infty}\xi'_{i,j} b^{-i}$, respectively, we define
$$
v(x_j, x'_j) := \min\{i - 1 \mid \xi_{i,j} \neq \xi'_{i,j}\},
$$
where we set $v(x_j, x'_j) = \infty$ if $x_j = x'_j$, and
$$
v(\bsx, \bsx') := \min_{1 \leq j \leq s} v(x_j, x'_j).
$$
Further for $\bsx \in [0,1]^s$ and $N \in \nat$, we define a cube $I_N(\bsx) := \{\bsx' \in [0,1]^s \colon v(\bsx, \bsx') \geq N\}$. Note that each edge of $I_N(\bsx)$ is of length $b^{-N}$.
We need the following lemma.
\begin{lemma}\label{lem:measure-zero}
Let $\bsx = (x_1, \dots, x_s), \bsx' = (x'_1, \dots, x'_s) \in [0,1]^s$.
Assume that $\bsx' \in I_{N+1}(\bsx)$.
Then the Lebesgue measure of the set
$$
S := \{
\bsz \in [0,1]^s \colon
\phi_b(\bsx' \oplus \bsz) \not\in I_N(\phi_b(\bsx \oplus \bsz))
\}
$$ is $0$.
\end{lemma}
\begin{proof}
Since
$S = \cup_{1 \leq j \leq s}
\{\bsz=(z_1, \dots, z_s) \colon v(\phi_b(x_j \oplus z_j), \phi_b(x'_j \oplus z_j)) < N \}$,
it suffices to show that
the set
$T := \{z\in[0,1] \colon v(\phi_b(x \oplus z), \phi_b(x' \oplus z)) < N  \}$
is countable for $x, x' \in [0,1]$ with $v(x, x') \geq N+1$.
If $x \oplus z$ and $\phi_b(x \oplus z)$ are not $b$-adic rationals
for $x = \sum_{i=1}^{\infty}\xi_{i} b^{-i}$ and $z = \sum_{i=1}^{\infty}\zeta_{i} b^{-i}$,
then the $b$-adic expansion of $\phi_b(x \oplus z)$ is uniquely represented by
$$
\phi_b(x \oplus z) = \sum_{i=1}^\infty \eta_ib^{-i} \quad \text{with}\quad \eta_i=(\xi_{i+1} + \zeta_{i+1}) -(\xi_1 + \zeta_1) \pmod b.
$$
It follows that if $x \oplus z$, $x' \oplus z$, $\phi_b(x \oplus z)$ and $\phi_b(x' \oplus z)$ are not $b$-adic rationals,
\begin{align*}
  v(\phi_b(x \oplus z), \phi_b(x' \oplus z)) & = v(x \oplus z, x' \oplus z)-1 \\
  & = v(x , x')-1 \\
  & \ge (N+1) -1 =N,
\end{align*}
which violates the condition of $T$.
Thus we have
\begin{align*}
[0,1] \backslash T
&\supset
\{z\colon x \oplus z \text{ is not a $b$-adic rational} \}
\cap
\{z\colon x' \oplus z \text{ is not a $b$-adic rational} \}\\
&\quad \cap
\{z\colon \phi_b(x \oplus z) \text{ is not a $b$-adic rational} \}\\
&\quad \cap
\{z\colon \phi_b(x' \oplus z) \text{ is not a $b$-adic rational} \}.
\end{align*}
The complement of every set on the right-hand side is countable.
Hence the set $T$ is countable.
\end{proof}
In the following lemma, we first give a sufficient condition on a function $f:[0,1]^s\to \RR$ such that the Walsh series expansion of $f$ converges to $f$ pointwise absolutely.

\begin{lemma}\label{lem: convergence of the Walsh expansion}
Let $f \colon [0,1]^s \to \RR$ be a function.
We assume that the following conditions hold:
\begin{enumerate}
\item For every $\bsx \in [0,1]^s$ and every $\epsilon > 0$,
there exists $N \in \nat$ such that $|f(\bsx) - f(\bsx')| \leq \epsilon$ holds
for any $\bsx' \in I_N(\bsx)$.
\item $\sum_{\bsk \in \nat_0^s} | \hat{f}(\bsk) | < \infty$.
\end{enumerate}
Then we have pointwise absolute convergence
$$
f(\bsx) = \sum_{\bsk \in \nat_0^s} \hat{f}(\bsk) \wal_{\bsk}(\bsx).
$$
\end{lemma}
\begin{proof}
By the second assumption, we have
$$
\sum_{\bsk \in \nat_0^s} |\hat{f}(\bsk) \wal_{\bsk}(\bsx)|
= \sum_{\bsk \in \nat_0^s} |\hat{f}(\bsk)| < \infty,
$$
and thus $\sum_{\bsk \in \nat_0^s} \hat{f}(\bsk) \wal_{\bsk}(\bsx)$
converges absolutely.
Therefore it suffices to show that
$\lim_{l \to \infty} \sum_{\bsk < b^l} \hat{f}(\bsk) \wal_{\bsk} (\bsx) = f(\bsx)$,
where $\bsk < b^l$ means that $k_j < b^l$ holds for every $j$.
In fact, we have
\begin{align*}
\sum_{\bsk < b^l} \hat{f}(\bsk) \wal_{\bsk} (\bsx)
& = \sum_{\bsk < b^l} \int_{\bsy \in [0,1]^s} f(\bsy) \overline{\wal_{\bsk}(\bsy)} \rd \bsy \cdot \wal_{\bsk}(\bsx)\\
& = \int_{\bsy \in [0,1]^s} f(\bsy) \sum_{\bsk < b^l} \wal_{\bsk} (\bsx \ominus \bsy) \rd \bsy \\
& = b^{ls} \int_{\bsy \in I_l(\bsx)} f(\bsy) \rd \bsy \\
& \to f(\bsx) \quad \mathrm{as}\quad l\to \infty ,
\end{align*}
where we use \cite[Lemma~A.17]{DP10} in the third equality and the first assumption implies the last convergence.
\end{proof}
\noindent
Based on Lemma~\ref{lem: convergence of the Walsh expansion}, we give a sufficient condition that the Walsh series expansion of $\Kcal_{\sh,\phi_b}$ converges to $\Kcal_{\sh,\phi_b}$ pointwise absolutely.
\begin{proposition}\label{prop:the Walsh expansion of kernel}
Let $\Kcal(\bsx, \bsy) \colon [0,1]^{2s} \to \RR$ be continuous.
If
$$
\sum_{\bsk \in \nat_0^s} |\hat{\Kcal}(\bsk, \bsk)| < \infty,
$$
then the Walsh expansion of $\Kcal_{\sh,\phi_b}$ converges to $\Kcal_{\sh,\phi_b}$
pointwise absolutely. In particular, we have
$$
\Kcal_{\sh,\phi_b}(\bsx,\bsy) =
\sum_{\bsk \in \nat_0^s} \hat{\Kcal}(\lfloor \bsk/b\rfloor, \lfloor \bsk/b\rfloor)\wal_{\bsk}(\bsx)\overline{\wal_{\bsk}(\bsy)}.
$$
\end{proposition}
\begin{proof}
By Theorem~\ref{thm:reproducing_kernel} we have
\begin{align*}
\sum_{\bsk, \bsl \in \nat_0^s} |\hat{\Kcal}_{\sh,\phi_b}(\bsk, \bsl)|
&\leq \sum_{\bsk \in \nat_0^s} |\hat{\Kcal}(\lfloor \bsk/b\rfloor, \lfloor \bsk/b\rfloor)|\\
&\leq b^{2s} \sum_{\bsk \in \nat_0^s} |\hat{\Kcal}(\bsk, \bsk)|
< \infty,
\end{align*}
and thus $\Kcal_{\sh,\phi_b}$ satisfies the second assumption of Lemma~\ref{lem: convergence of the Walsh expansion}.
Therefore it suffices to show that $\Kcal_{\sh,\phi_b}$ satisfies the first assumption of Lemma~\ref{lem: convergence of the Walsh expansion}.
Let $\bsx, \bsy \in [0,1]^s$ and $\epsilon > 0$.
Since $\Kcal$ is continuous on a compact topological space $[0,1]^{2s}$,
$\Kcal$ is uniformly continuous.
Hence there exists a positive integer $N$ such that
$|\Kcal(\bsx', \bsy') - \Kcal(\bsx'', \bsy'')| < \epsilon$ holds
for any $\bsx', \bsx'', \bsy', \bsy'' \in [0,1]^s$
with $(\bsx'', \bsy'') \in I_{N}((\bsx', \bsy'))$.
Suppose that $(\bsx', \bsy') \in I_{N + 1}((\bsx, \bsy))$.
By Lemma~\ref{lem:measure-zero}, the Lebesgue measure of the set 
$\{\bssigma \in [0,1]^s \colon (\phi_b(\bsx' \oplus \bssigma), \phi_b(\bsy' \oplus \bssigma))
\not\in I_N((\phi_b(\bsx \oplus \bssigma), \phi_b(\bsy \oplus \bssigma)))\}$ is $0$.
Therefore we have
\begin{align*}
& |\Kcal_{\sh,\phi_b}(\bsx, \bsy) - \Kcal_{\sh,\phi_b}(\bsx', \bsy')| \\
= & \left| \int_{[0,1]^{s}}\Kcal(\phi_b(\bsx \oplus \bssigma), \phi_b(\bsy \oplus \bssigma)) -	\Kcal(\phi_b(\bsx' \oplus \bssigma), \phi_b(\bsy' \oplus \bssigma))	\rd \bssigma \right| \\
\leq & \int_{[0,1]^{s}}
|\Kcal(\phi_b(\bsx \oplus \bssigma), \phi_b(\bsy \oplus \bssigma)) -
	\Kcal(\phi_b(\bsx' \oplus \bssigma), \phi_b(\bsy' \oplus \bssigma))|
	\rd \bssigma \\
\leq & \int_{[0,1]^{s}} \epsilon \rd \bssigma = \epsilon. \qedhere
\end{align*}
\end{proof}

\begin{remark}
Although Cristea et al.~\cite{CDLP07} skip the above argument, Proposition \ref{prop:the Walsh expansion of kernel} gives a sufficient condition to have Theorem \ref{thm:mse_formula} and \cite[Theorem~3]{CDLP07} also, since we have to suppose that the Walsh series expansion of $\Kcal_{\sh,\phi_b}$ converges to $\Kcal_{\sh,\phi_b}$ pointwise absolutely.

We further note that there is a proof that the digital shift invariant kernel is equal to its Walsh expansion with some condition in \cite[Lemma~12.2]{DP10}. The statement of the lemma itself is correct, although the proof uses the following incorrect claim. For $\bsx=(x_1,\ldots,x_s),\bsy=(y_1,\ldots,y_s)\in [0,1)^s$ with $x_i=\xi_{i,1}b^{-1}+\xi_{i,2}b^{-2}+\cdots$ and $y_i=\eta_{i,1}b^{-1}+\eta_{i,2}b^{-2}+\cdots$, if $\max_{1\le i\le s}|x_i-y_i|<b^{-a}$ for some $a\in \nat$, it follows that $\xi_{i,k}=\eta_{i,k}$ for all $1\le i\le s$ and $1\le k\le a$. However, this is {\em not} always the case. A counterexample is given by setting
$$\xi_{i,1}=1,\quad \xi_{i,2}=\xi_{i,3}=\cdots=0,$$ and $$\eta_{i,1}=0,\quad \eta_{i,2}=\cdots = \eta_{i,a+1}=b-1,\quad \eta_{i,a+2}=\eta_{i,a+3}=\cdots=0,$$ for all $1\le i\le s$. For such $\bsx,\bsy$, we have $\max_{1\le i\le s}|x_i-y_i|<b^{-a}$, but {\em not} $\xi_{i,k}=\eta_{i,k}$ for all $1\le i\le s$ and $1\le k\le a$. Hence, the continuity of the reproducing kernel does {\em not} necessarily imply the continuity of its digital shift invariant kernel in the usual topology of $[0,1]^s$. The proof can be fixed in a way similar to our proof above.

\end{remark}

\section{Existence of good higher order polynomial lattice rules in unanchored Sobolev spaces}\label{sec:hopoly}
Finally, in this section, we consider unanchored Sobolev spaces of smoothness of arbitrary high order $\alpha\ge 2$ as $\Hcal$, denoted by $\Hcal_{\alpha,\bsgamma}$, and higher order polynomial lattice point sets over $\ZZ_b$ as $P_{b^m,s}$ for a prime $b$. Our goal here is to prove the existence of good higher order polynomial lattice rules which achieve almost the optimal rate of the mean square worst-case error in these spaces when $n\ge \alpha m/2$. As shown in \cite{S63}, we cannot achieve the convergence rate of the mean square worst-case error of order $b^{-2\alpha m}$ in $\Hcal_{\alpha,\bsgamma}$. Thus, the convergence rate of order $b^{-2\alpha m+\epsilon}$ with arbitrary small $\epsilon>0$ is almost optimal.

\subsection{A bound on the mean square worst-case error}
First we follow the expositions of \cite{BD09,DP07} to introduce the reproducing kernel Hilbert space $\Hcal_{\alpha,\bsgamma}$ that we consider in this section. Let $\alpha$ be a positive integer greater than 1 and let $\bsgamma=(\gamma_u)_{u\subseteq \{1,\dots,s\}}$ be a set of non-negative numbers.

Here $\bsgamma$ are called {\em weights} and play a major role in moderating the importance of different variables or groups of variables in the space $\Hcal_{\alpha,\bsgamma}$ and also in analyzing the information complexity that is defined as the minimum number of points $N(\varepsilon,s)$ required to reduce the initial error by a factor $\varepsilon\in (0,1)$, see \cite{SW98}. Our particular interest is to give sufficient conditions on the weights when the bound on $N(\varepsilon,s)$ does not depend on the dimension, or does depend only polynomially on the dimension, see Corollary \ref{cor:tractability}.

For $\bsx,\bsy\in [0,1]^s$, the reproducing kernel of $\Hcal_{\alpha,\bsgamma}$ is given by
  \begin{align}\label{eq:Sobolev_kernel}
    \Kcal_{\alpha,\bsgamma}(\bsx,\bsy)= & \sum_{u\subseteq \{1,\ldots,s\}}\gamma_u \prod_{j\in u}\Kcal_{\alpha,(1)}(x_j,y_j) \nonumber \\
    = & \sum_{u\subseteq \{1,\ldots,s\}}\gamma_u \prod_{j\in u}\left( \sum_{\tau=1}^{\alpha}\frac{B_{\tau}(x_j)B_{\tau}(y_j)}{(\tau !)^2}+(-1)^{\alpha+1}\frac{B_{2\alpha}(|x_j-y_j|)}{(2\alpha)!}\right) ,
  \end{align}
where $B_{\tau}$ denotes the Bernoulli polynomial of degree $\tau$, and $x_j$ and $y_j$ denote the $j$-th coordinates of $\bsx$ and $\bsy$, respectively. In the following, for $f\in \Hcal_{\alpha,\bsgamma}$ and a vector $(\alpha_1,\dots,\alpha_s)\in \nat_0^s$ with $0\le \alpha_j\le \alpha$ for all $j$, we denote by $f^{(\alpha_1,\dots,\alpha_s)}$ the partial mixed derivative of $f$ of $(\alpha_1,\dots,\alpha_s)$-th order. Then, for $f,g\in \Hcal_{\alpha,\bsgamma}$, the inner product is defined as
  \begin{align*}
    \langle f,g\rangle_{\Hcal_{\alpha,\bsgamma}} = & \sum_{u\subseteq \{1,\ldots,s\}}\gamma_u^{-1}\sum_{v\subseteq u}\sum_{\bstau_{u\setminus v}\in \{1,\ldots, \alpha-1\}^{|u\setminus v|}}\\
    & \int_{[0,1]^{|v|}} \left(\int_{[0,1]^{s-|v|}}f^{(\bstau_{u\setminus v},\bsalpha_v,\bszero)}(\bsx)\rd\bsx_{-v}\right) \\
    & \times \left(\int_{[0,1]^{s-|v|}}g^{(\bstau_{u\setminus v},\bsalpha_v,\bszero)}(\bsx)\rd\bsx_{-v}\right)\rd\bsx_{v} ,
  \end{align*}
where we use the following notation: For $\bstau_{u\setminus v}=(\tau_j)_{j\in u\setminus v}$, we denote by $(\bstau_{u\setminus v},\bsalpha_v,\bszero)$ the vector in which the $j$-th component is $\tau_j$ for $j\in u\setminus v$, $\alpha$ for $j\in v$, and 0 for $\{1,\ldots,s\}\setminus u$. For $v\subseteq \{1,\ldots,s\}$, we simply write $-v:=\{1,\ldots,s\}\setminus v$, $\bsx_{v}=(x_j)_{j\in v}$ and $\bsx_{-v}=(x_j)_{j\in -v}$. As in \cite{BD09}, for $u\subseteq \{1,\ldots,s\}$ such that $\gamma_u=0$, we assume that the corresponding inner double sum equals 0 and we set $0/0=0$.

We now consider the Walsh series expansion of $\Kcal_{\alpha,\bsgamma}$,
  \begin{align*}
     \Kcal_{\alpha,\bsgamma}(\bsx,\bsy) = & \sum_{\bsk,\bsl\in \nat_0^s}\hat{\Kcal}_{\alpha,\bsgamma}(\bsk,\bsl)\wal_{\bsk}(\bsx)\overline{\wal_{\bsl}(\bsy)} \\
     = & \sum_{u,v\subseteq \{1,\dots,s\}}\sum_{\bsk_u\in \nat^{|u|}}\sum_{\bsl_v\in \nat^{|v|}}\hat{\Kcal}_{\alpha,\bsgamma}((\bsk_u,\bszero),(\bsl_v,\bszero))\wal_{(\bsk_u,\bszero)}(\bsx)\overline{\wal_{(\bsl_v,\bszero)}(\bsy)},
  \end{align*}
where we write $(\bsk_u,\bszero)$ for the $s$-dimensional vector whose $j$-th component is $k_j$ if $j\in u$ and zero otherwise, and we use the same notation for $(\bsl_v,\bszero)$. According to \cite[Section~3]{BD09}, we have the following.

\begin{lemma}
For $u,v\subseteq \{1,\dots,s\}$, $\bsk_u\in \nat^{|u|}$ and $\bsl_v\in \nat^{|v|}$, the $((\bsk_u,\bszero),(\bsl_v,\bszero))$-th Walsh coefficient is given by
  \begin{align*}
     \hat{\Kcal}_{\alpha,\bsgamma}((\bsk_u,\bszero),(\bsl_v,\bszero)) = & \left\{ \begin{array}{ll}
     \gamma_u \prod_{j\in u}\hat{\Kcal}_{\alpha,(1)}(k_j,l_j) & \text{if} \quad u=v ,  \\
     0 & \text{otherwise} .  \\
    \end{array} \right.
  \end{align*}
A bound on the Walsh coefficients $\hat{\Kcal}_{\alpha,(1)}(k,l)$ for $k,l\in \nat$ is given as
  \begin{align*}
     \left| \hat{\Kcal}_{\alpha,(1)}(k,l) \right| \le C_{\alpha,b}b^{-\mu_{\alpha}(k)-\mu_{\alpha}(l)}  ,
  \end{align*}
where the constant $C_{\alpha,b}$ depends only on $\alpha$ and $b$, and $\mu_{\alpha}(k)$ is defined, for $k\in \nat$, as
  \begin{align*}
     \mu_{\alpha}(k)=a_1+\dots+a_{\min(v,\alpha)},
  \end{align*}
where we denote the $b$-adic expansion of $k$ by $k=\kappa_1b^{a_1-1}+\dots+\kappa_vb^{a_v-1}$ with $0< \kappa_1,\dots,\kappa_v<b$ and $a_1>\dots>a_v>0$.
\end{lemma}

Here $\Kcal_{\alpha,\bsgamma}$ is continuous in the usual topology of $[0,1]^s$. Therefore,  we have $\int_{[0,1]^s}\sqrt{\Kcal_{\alpha,\bsgamma}(\bsx,\bsx)}\rd \bsx<\infty$, and $\sum_{\bsk \in \nat_0^s} |\hat{\Kcal}_{\alpha,\bsgamma}(\bsk, \bsk)|$ is also finite since
  \begin{align*}
     \sum_{\bsk \in \nat_0^s} |\hat{\Kcal}_{\alpha,\bsgamma}(\bsk, \bsk)| \le & \sum_{u\subseteq \{1,\dots,s\}}\gamma_u C^{|u|}_{\alpha,b}\sum_{\bsk_u\in \nat^{|u|}}\prod_{j\in u}b^{-2\mu_{\alpha}(k_j)} \\
     \le & \sum_{u\subseteq \{1,\dots,s\}}\gamma_u C^{|u|}_{\alpha,b}\left( \sum_{k=1}^{\infty}b^{-2\mu_{1}(k)}\right)^{|u|} \\
     = & \sum_{u\subseteq \{1,\dots,s\}}\gamma_u C^{|u|}_{\alpha,b}\left( \sum_{l=1}^{\infty}\sum_{k=b^{l-1}}^{b^l-1}b^{-2l}\right)^{|u|} \\
     = & \sum_{u\subseteq \{1,\dots,s\}}\gamma_u C^{|u|}_{\alpha,b}b^{-|u|} < \infty .
  \end{align*}
Thus, from Proposition \ref{prop:the Walsh expansion of kernel}, the Walsh expansion of the folded digitally shifted reproducing kernel of $\Kcal_{\alpha,\bsgamma}$, denoted by $\Kcal_{\alpha,\bsgamma,\sh,\phi_b}$, converges to $\Kcal_{\alpha,\bsgamma,\sh,\phi_b}$ pointwise absolutely.

Combining this result with our Theorem \ref{thm:mse_formula}, we have the following theorem that shows a bound on the mean square worst-case error for digital nets over $\ZZ_b$ which are randomly digitally shifted and then folded using the $b$-TT.

\begin{theorem}\label{thm:mse_sobolev_bound}
Let $\Kcal_{\alpha,\bsgamma}$ be the reproducing kernel described as in (\ref{eq:Sobolev_kernel}). Let $P_{b^m,s}$ be a digital net over $\ZZ_b$ and $\phi_b(P_{b^m,s}\oplus \bssigma)$ be a point set obtained by digitally shifting $P_{b^m,s}$ by a randomly chosen $\bssigma$ and then folding using the $b$-TT. Then the mean square worst-case error of $\phi_b(P_{N,s}\oplus \bssigma)$ is bounded by
    \begin{align*}
     \hat{e}^2(P_{b^m,s},\Kcal_{\alpha,\bsgamma}) \le \sum_{\emptyset \ne u\subseteq \{1,\dots,s\}}\gamma_u C_{\alpha,b}^{|u|}\sum_{\substack{\bsk_u\in \Ecal_{b}^{|u|} \\ (\bsk_u,\bszero)\in D^{\perp}(P_{b^m,s})}}b^{-2\mu_{\alpha}(\lfloor \bsk_u/b\rfloor)} =: B_{\alpha,\bsgamma}(P_{b^m,s}),
  \end{align*}
where $D^{\perp}(P_{b^m,s})$ is the dual net of $P_{b^m,s}$ and $\mu_{\alpha}(\lfloor \bsk_u/b\rfloor)=\sum_{j\in u}\mu_{\alpha}(\lfloor k_j/b\rfloor)$.
\end{theorem}

\subsection{Existence result}
In this subsection, we focus on higher order polynomial lattice point sets over $\ZZ_b$ as $P_{b^m,s}$. We denote the bound shown in Theorem \ref{thm:mse_sobolev_bound} as $B_{\alpha,\bsgamma}(\bsq,p)$ to emphasize the role of $\bsq$ and $p$ in higher order polynomial lattice point sets, see Definition \ref{def:hopoly}. Here we prove the existence of good higher order polynomial lattice rules which achieve almost the optimal convergence rate in $\Hcal_{\alpha,\bsgamma}$ when $n\ge \alpha m/2$. Without loss of generality, we can restrict ourselves to considering a set of polynomials $\bsq=(q_1,\ldots,q_s)\in G_{b,n}^s$ for $1\le j\le s$, where
  \begin{align*}
    G_{b,n}=\{q\in \ZZ_b[x] \colon \deg(q)<n \} .
  \end{align*}
This implies that we have $b^{ns}$ candidates for $\bsq$ in total. The existence result shows that there exists at least one good generating vector $\bsq$ among them. The following theorem is exactly what we want.

\begin{theorem}\label{theorem:existence}
Let $p\in \FF_b[x]$ be an irreducible polynomial with $\deg(p)=n$. Then, for $\alpha \ge 2$, there exists at least one vector of polynomials $\bsq\in G_{b,n}^s$, such that for the higher order polynomial lattice point set with generating vector $\bsq$ and modulus $p$ we have
  \begin{align*}
    B_{\alpha,\bsgamma}(\bsq,p) \le \frac{1}{b^{\min(m/\lambda, 4n)}}\left[\sum_{\emptyset \ne u\subseteq \{1,\ldots,s\}}\gamma_u^{\lambda}C_{\alpha,b}^{\lambda|u|}\left( A_{\alpha,b,\lambda,1}^{|u|}+A_{\alpha,b,\lambda,2}^{|u|}\right)\right]^{1/\lambda} ,
  \end{align*}
for any $1/(2\alpha)<\lambda \le 1$, where
  \begin{align*}
    A_{\alpha,b,\lambda,1}=\frac{b}{b-1}\left[ \sum_{v=1}^{\alpha-1}\prod_{i=1}^{v}\left( \frac{b-1}{b^{2\lambda i}-1}\right)+\frac{b^{2\lambda \alpha}-1}{b^{2\lambda \alpha}-b}\prod_{i=1}^{\alpha}\left( \frac{b-1}{b^{2\lambda i}-1}\right)\right] ,
  \end{align*}
and
  \begin{align*}
    A_{\alpha,b,\lambda,2}=\frac{1}{b-1}\sum_{v=2}^{\alpha-1}\prod_{i=1}^{v}\left( \frac{b^{2\lambda}(b-1)}{b^{2\lambda i}-1}\right)+\frac{b^{2\lambda}}{b^{2\lambda\alpha}-b}\prod_{i=1}^{\alpha-1}\left( \frac{b^{2\lambda}(b-1)}{b^{2\lambda i}-1}\right) .
  \end{align*}
\end{theorem}
\noindent
In order to prove the above theorem, we need the following lemmas.

\begin{lemma}\label{lem:sum-of-digit1}
Let $b$ be a prime. For $v\in \nat$, define
  \begin{align*}
    N_b(v) := | \{\kappa_1,\dots,\kappa_v\in \ZZ_b\setminus \{0\} \colon \kappa_1+\dots+\kappa_v\equiv 0 \pmod b\}| .
  \end{align*}
Then we have
  \begin{align*}
    N_b(v) \le \left\{ \begin{array}{ll}
     0 & \text{if}\quad v = 1,  \\
     (b-1)^{v-1} & \text{if}\quad v > 1.  \\
    \end{array} \right.
  \end{align*}
\end{lemma}

\begin{proof}
For $v=1$, $\kappa_1\in \{1,2,\dots,b-1\}$ cannot be a multiple of $b$. Hence we have $N_b(1)=0$.

We now suppose $v>1$. Since $\kappa_v\in \{1,2,\dots,b-1\}$, $\kappa_1+\dots+\kappa_{v-1}\not\equiv 0\pmod b$. To put it the other way around, if $\kappa_1,\dots,\kappa_{v-1}$ are given such that $\kappa_1+\dots+\kappa_{v-1}\not\equiv 0\pmod b$, then we have exactly one choice $\kappa_v\in \{1,2,\dots,b-1\}$ which satisfies $\kappa_1+\dots+\kappa_v\equiv 0 \pmod b$. Thus we have
  \begin{align*}
    N_b(v) = & | \{\kappa_1,\dots,\kappa_{v-1}\in \ZZ_b\setminus \{0\} \colon \kappa_1+\dots+\kappa_{v-1}\not\equiv 0 \pmod b\}| \\
    = & | \{\kappa_1,\dots,\kappa_{v-1}\in \ZZ_b\setminus \{0\}\}| \\
    & - | \{\kappa_1,\dots,\kappa_{v-1}\in \ZZ_b\setminus \{0\} \colon \kappa_1+\dots+\kappa_{v-1}\equiv 0 \pmod b\}| \\
    = & (b-1)^{v-1}-N_b(v-1) \le (b-1)^{v-1}.
  \end{align*}
\end{proof}

\begin{lemma}\label{lem:sum-of-digit2}
Let $b$ be a prime, $\alpha\ge 2$ be an integer and $\lambda >1/(2\alpha)$ be a real number. Let $A_{\alpha,b,\lambda,1}$ and $A_{\alpha,b,\lambda,2}$ be given as in Theorem \ref{theorem:existence}.
\begin{enumerate}
\item We have 
  \begin{align*}
    \sum_{k\in \Ecal_b}b^{-2\lambda\mu_{\alpha}(\lfloor k/b\rfloor)}\le A_{\alpha,b,\lambda,1}.
  \end{align*}
\item For $n\in \nat$, we have 
  \begin{align*}
    \sum_{\substack{k\in \Ecal_b\\ b^n\mid k}}b^{-2\lambda\mu_{\alpha}(\lfloor k/b\rfloor)}\le \frac{A_{\alpha,b,\lambda,2}}{b^{4\lambda n}} .
  \end{align*}
\end{enumerate}
\end{lemma}

\begin{proof}
Let us consider the first part of the lemma. We consider the $b$-adic expansion of $k\in \Ecal_b$ of the form $k=\kappa_1b^{a_1-1}+\dots+\kappa_vb^{a_v-1}$ with $0< \kappa_1,\dots,\kappa_v<b$ and $a_1>\dots>a_v>0$, and arrange every element of $\Ecal_b$ according to the value of $v$ in their expansions. Since the choice of $\kappa_1,\dots,\kappa_v$ does not change $\mu_{\alpha}(k)$, we have
  \begin{align*}
    \sum_{k\in \Ecal_b}b^{-2\lambda\mu_{\alpha}(\lfloor k/b\rfloor)}=\sum_{v=1}^{\infty}N_b(v)\sum_{0<a_v <\dots <a_1}b^{-2\lambda \mu_{\alpha}(\lfloor (b^{a_1-1}+\dots +b^{a_{v}-1})/b\rfloor)} ,
  \end{align*}
where $N_b(v)$ is defined as in Lemma \ref{lem:sum-of-digit1}. Using the result of Lemma \ref{lem:sum-of-digit1} and considering the cases $a_v=1$ and $a_v>1$ separately, we have
  \begin{align*}
    \sum_{k\in \Ecal_b}b^{-2\lambda\mu_{\alpha}(\lfloor k/b\rfloor)} \le & \sum_{v=2}^{\infty}(b-1)^{v-1}\sum_{0<a_v<\dots <a_1}b^{-2\lambda \mu_{\alpha}(\lfloor (b^{a_1-1}+\dots +b^{a_v-1})/b\rfloor)}  \\
    = & \sum_{v=2}^{\infty}(b-1)^{v-1}\sum_{1<a_{v-1}<\dots <a_1}b^{-2\lambda \mu_{\alpha}(b^{a_1-2}+\dots +b^{a_{v-1}-2})} \\
    & + \sum_{v=2}^{\infty}(b-1)^{v-1}\sum_{1<a_v<\dots <a_1}b^{-2\lambda \mu_{\alpha}(b^{a_1-2}+\dots +b^{a_v-2})} \\
    = & \sum_{v=1}^{\infty}(b-1)^v\sum_{0<a_v<\dots <a_1}b^{-2\lambda \mu_{\alpha}(b^{a_1-1}+\dots +b^{a_v-1})} \\
    & + \frac{1}{b-1}\sum_{v=2}^{\infty}(b-1)^v\sum_{0<a_v<\dots <a_1}b^{-2\lambda \mu_{\alpha}(b^{a_1-1}+\dots +b^{a_v-1})} \\
    \le & \frac{b}{b-1}\sum_{v=1}^{\infty}(b-1)^v\sum_{0<a_v<\dots <a_1}b^{-2\lambda \mu_{\alpha}(b^{a_1-1}+\dots +b^{a_v-1})} .
  \end{align*}
From the definition of $\mu_{\alpha}$ we have
  \begin{align*}
     \mu_{\alpha}(b^{a_1-1}+\cdots +b^{a_{v}-1})= \left\{ \begin{array}{ll}
     a_1+\dots+a_v & \text{if}\quad v < \alpha,  \\
     a_1+\dots+a_{\alpha} & \text{if}\quad v \ge \alpha .  \\
    \end{array} \right.
  \end{align*}
In the following, we split the last infinite sum over $v$ into one finite sum from $v=1$ to $v=\alpha-1$ and the other infinite sum over $v\ge \alpha$ to obtain
  \begin{align}\label{eq:sum-of-digit_pf1}
    \sum_{k\in \Ecal_b}b^{-2\lambda\mu_{\alpha}(\lfloor k/b\rfloor)} \le & \frac{b}{b-1}\sum_{v=1}^{\alpha-1}(b-1)^v\sum_{0<a_v<\dots <a_1}b^{-2\lambda (a_1+\dots +a_v)} \nonumber \\
    & + \frac{b}{b-1}\sum_{v=\alpha}^{\infty}(b-1)^v\sum_{0<a_v<\dots <a_1}b^{-2\lambda (a_1+\dots +a_{\alpha})} .
  \end{align}
We now follow a way analogous to the proof of \cite[Lemma~3.1]{BDGP11}. For the first term on the right-hand side of (\ref{eq:sum-of-digit_pf1}), we have
  \begin{align*}
    \sum_{0<a_v<\cdots <a_1}b^{-2\lambda (a_1+\cdots +a_v)} & = \sum_{a_v=1}^{\infty}b^{-2\lambda a_v}\sum_{a_{v-1}=a_v+1}^{\infty}b^{-2\lambda a_{v-1}}\dots \sum_{a_1=a_2+1}^{\infty}b^{-2\lambda a_1} \\
    & = \frac{1}{b^{2\lambda}-1}\sum_{a_v=1}^{\infty}b^{-2\lambda a_v}\sum_{a_{v-1}=a_v+1}^{\infty}b^{-2\lambda a_{v-1}}\dots \sum_{a_2=a_3+1}^{\infty}b^{-4\lambda a_2} \\
    & \vdots \\
    & = \prod_{i=1}^{v-1}\left( \frac{1}{b^{2\lambda i}-1}\right)\sum_{a_v=1}^{\infty}b^{-2\lambda va_v} = \prod_{i=1}^{v}\left( \frac{1}{b^{2\lambda i}-1}\right) .
  \end{align*}
For the second term on the right-hand side of (\ref{eq:sum-of-digit_pf1}), we have
  \begin{align*}
    & \sum_{v=\alpha}^{\infty}(b-1)^{v}\sum_{0<a_v<\dots <a_1}b^{-2\lambda (a_1+\dots +a_{\alpha})} \\
    = & \sum_{v=\alpha}^{\infty}(b-1)^{v}\sum_{a_v=1}^{\infty}\sum_{a_{v-1}=a_v+1}^{\infty}\dots \sum_{a_{\alpha+1}=a_{\alpha+2}+1}^{\infty}\sum_{a_{\alpha}=a_{\alpha+1}+1}^{\infty}b^{-2\lambda a_{\alpha}}\dots \sum_{a_1=a_2+1}^{\infty}b^{-2\lambda a_1} \\
    = & \prod_{i=1}^{\alpha}\left( \frac{1}{b^{2\lambda i}-1}\right)\sum_{v=\alpha}^{\infty}(b-1)^{v}\sum_{a_v=1}^{\infty}\sum_{a_{v-1}=a_v+1}^{\infty}\dots \sum_{a_{\alpha+1}=a_{\alpha+2}+1}^{\infty}b^{-2\lambda \alpha a_{\alpha+1}} \\
    = & \prod_{i=1}^{\alpha}\left( \frac{b-1}{b^{2\lambda i}-1}\right)\sum_{v=\alpha}^{\infty}(b-1)^{v-\alpha}\left( \frac{1}{b^{2\lambda \alpha}-1}\right)^{v-\alpha} = \frac{b^{2\lambda \alpha}-1}{b^{2\lambda \alpha}-b}\prod_{i=1}^{\alpha}\left( \frac{b-1}{b^{2\lambda i}-1}\right) ,
  \end{align*}
where the last equality requires $\lambda > 1/(2\alpha)$. Substituting these results to (\ref{eq:sum-of-digit_pf1}), the result for the first part follows.

Let us move on to the second part of the lemma. If $b^n\mid k$, $k$ is of the form $lb^{n}$ for $l\in \nat$. Using the identity $\delta_b(lb^{n})=\delta_b(l)$, we have
  \begin{align*}
    \sum_{\substack{k\in \Ecal_b\\ b^n\mid k}}b^{-2\lambda\mu_{\alpha}(\lfloor k/b\rfloor)} = \sum_{l\in \Ecal_b}b^{-2\lambda\mu_{\alpha}(lb^{n-1})} .
  \end{align*}
As in the first part of this lemma, we consider the $b$-adic expansion of $l\in \Ecal_b$ of the form $l=\kappa_1b^{a_1-1}+\dots+\kappa_vb^{a_v-1}$ with $0< \kappa_1,\dots,\kappa_v<b$ and $a_1>\dots>a_v>0$, and partition $\Ecal_b$ according to the value of $v$ in their expansions. Using the result of Lemma \ref{lem:sum-of-digit1} and splitting the infinite sum over $v$ into one finite sum from $v=2$ to $v=\alpha-1$ and the other infinite sum over $v\ge \alpha$, we have 
  \begin{align}\label{eq:sum-of-digit_pf2}
    \sum_{\substack{k\in \Ecal_b\\ b^n\mid k}}b^{-2\lambda\mu_{\alpha}(\lfloor k/b\rfloor)} = & \sum_{v=1}^{\infty}N_b(v)\sum_{0<a_v<\cdots <a_1}b^{-2\lambda\mu_{\alpha}(b^{a_1+n-2}+\cdots +b^{a_v+n-2})} \nonumber \\
    \le & \sum_{v=2}^{\alpha-1}(b-1)^{v-1}\sum_{0<a_v<\cdots <a_1}b^{-2\lambda ((a_1+n-1)+\cdots +(a_v+n-1))} \nonumber \\
    & + \sum_{v=\alpha}^{\infty}(b-1)^{v-1}\sum_{0<a_v<\cdots <a_1}b^{-2\lambda ((a_1+n-1)+\cdots +(a_{\alpha}+n-1))} .
  \end{align}
We have for the first term on the rightmost side of (\ref{eq:sum-of-digit_pf2})
  \begin{align*}
    & \sum_{0<a_v<\cdots <a_1}b^{-2\lambda ((a_1+n-1)+\cdots +(a_v+n-1))} \\
    = & \frac{1}{b^{2\lambda v(n-1)}}\sum_{0<a_v<\cdots <a_1}b^{-2\lambda (a_1+\cdots +a_v)} \\
    = & \frac{1}{b^{2\lambda v(n-1)}}\prod_{i=1}^{v}\left( \frac{1}{b^{2\lambda i}-1}\right) \\
    \le & \frac{1}{b^{4\lambda n}}\prod_{i=1}^{v}\left( \frac{b^{2\lambda}}{b^{2\lambda i}-1}\right) ,
  \end{align*}
where we use, in the second equality, the result that appeared in the proof of the first part of this lemma, and the last inequality stems from the fact $v\ge 2$. As for the second term on the rightmost side of (\ref{eq:sum-of-digit_pf2}), we have
  \begin{align*}
    & \sum_{v=\alpha}^{\infty}(b-1)^{v-1}\sum_{0<a_v<\cdots <a_1}b^{-2\lambda ((a_1+n-1)+\cdots +(a_{\alpha}+n-1))} \\
    = & \sum_{v=\alpha}^{\infty}\frac{(b-1)^{v-1}}{b^{2\lambda \alpha (n-1)}}\sum_{a_v=1}^{\infty}\sum_{a_{v-1}=a_v+1}^{\infty}\dots \sum_{a_{\alpha+1}=a_{\alpha+2}+1}^{\infty}\sum_{a_{\alpha}=a_{\alpha+1}+1}^{\infty}b^{-2\lambda a_{v-1}}\dots \sum_{a_1=a_2+1}^{\infty}b^{-2\lambda a_1} \\
    = & \prod_{i=1}^{\alpha}\left( \frac{1}{b^{2\lambda i}-1}\right)\sum_{v=\alpha}^{\infty}\frac{(b-1)^{v-1}}{b^{2\lambda \alpha (n-1)}}\sum_{a_v=1}^{\infty}\sum_{a_{v-1}=a_v+1}^{\infty}\dots \sum_{a_{\alpha+1}=a_{\alpha+2}+1}^{\infty}b^{-2\lambda \alpha a_{\alpha+1}} \\
    = & \frac{(b-1)^{\alpha-1}}{b^{2\lambda \alpha (n-1)}}\prod_{i=1}^{\alpha}\left( \frac{1}{b^{2\lambda i}-1}\right)\sum_{v=\alpha}^{\infty}(b-1)^{v-\alpha}\left( \frac{1}{b^{2\lambda \alpha}-1}\right)^{v-\alpha} \\
    = & \frac{b^{2\lambda \alpha}-1}{b^{2\lambda \alpha}-b}\cdot\frac{(b-1)^{\alpha-1}}{b^{2\lambda \alpha (n-1)}}\prod_{i=1}^{\alpha}\left( \frac{1}{b^{2\lambda i}-1}\right) \\
    \le & \frac{1}{b^{4\lambda n}}\frac{b^{2\lambda}}{b^{2\lambda\alpha}-b}\prod_{i=1}^{\alpha-1}\left( \frac{b^{2\lambda}(b-1)}{b^{2\lambda i}-1}\right) ,
  \end{align*}
where we have the last inequality since $\alpha\ge 2$. Substituting these results into (\ref{eq:sum-of-digit_pf2}), the result for the second part follows.
\end{proof}
\noindent
Now we are ready to prove Theorem \ref{theorem:existence}. In the following proof, we shall use the following inequality that is sometimes
referred to as Jensen's inequality. For a sequence $(a_n)_{n\in \nat}$ of non-negative real numbers, we have
  \begin{align}\label{eq:jensen}
    \left( \sum_{n}a_n\right)^{\lambda} \le \sum_{n}a_n^{\lambda} ,
  \end{align}
for $0<\lambda \le 1$.

\begin{proof}[Proof of Theorem~\ref{theorem:existence}]
Due to an averaging argument, there exists at least one set of polynomials $\bsq\in G_{b,n}^s$ for which $B_{\alpha,\bsgamma}^{\lambda}(\bsq,p)$ is smaller than or equal to the average of $B_{\alpha,\bsgamma}^{\lambda}(\tilde{\bsq},p)$ over $\tilde{\bsq}\in G_{b,n}^s$ for any $1/(2\alpha)<\lambda\le 1$. That is,
  \begin{align}\label{eq:averaging}
    B_{\alpha,\bsgamma}^{\lambda}(\bsq,p) \le \frac{1}{b^{ns}}\sum_{\tilde{\bsq}\in G_{b,n}^s}B_{\alpha,\bsgamma}^{\lambda}(\tilde{\bsq},p)=:\bar{B}_{\alpha,\bsgamma,\lambda} .
  \end{align}
Applying the inequality (\ref{eq:jensen}), we have
  \begin{align*}
     \bar{B}_{\alpha,\bsgamma,\lambda} \le & \frac{1}{b^{ns}}\sum_{\tilde{\bsq}\in G_{b,n}^s}\sum_{\emptyset \ne u\subseteq \{1,\ldots,s\}}\gamma_u^{\lambda}C_{\alpha,b}^{\lambda|u|}\sum_{\substack{\bsk_u\in \Ecal_b^{|u|}\\ (\bsk_u,\bszero) \in \Dcal^{\perp}(\tilde{\bsq},p)}}b^{-2\lambda\mu_{\alpha}(\lfloor\bsk_u/b\rfloor)} \\
    = & \sum_{\emptyset \ne u\subseteq \{1,\ldots,s\}}\gamma_u^{\lambda}C_{\alpha,b}^{\lambda|u|}\sum_{\bsk_u\in \Ecal_b^{|u|}}b^{-2\lambda\mu_{\alpha}(\lfloor\bsk_u/b\rfloor)}\frac{1}{b^{n|u|}}\sum_{\substack{\tilde{\bsq}_u\in G_{b,n}^{|u|}\\ \rtr_n(\bsk_u)\cdot \tilde{\bsq}_u\equiv a \pmod p\\ \deg(a)<n-m}}1 ,
  \end{align*}
where we denote $\rtr_n(\bsk_u)\cdot \tilde{\bsq}_u=\sum_{j\in u}\rtr_n(k_j)\tilde{q}_j$. The innermost sum equals the number of solutions $\tilde{\bsq}_u\in G_{b,n}^{|u|}$ such that $\rtr_n(\bsk_u)\cdot \tilde{\bsq}_u\equiv a \pmod p$ with $\deg(a)<n-m$. If $\rtr_n(k_j)$ is a multiple of $p$ for all $j\in u$, we always have $\rtr_n(\bsk_u)\cdot \tilde{\bsq}_u\equiv 0 \pmod p$ independently of $\tilde{\bsq}_u$. Otherwise if there exists at least one component $\rtr_n(k_j)$ which is not a multiple of $p$, then there are $b^{n-m}$ possible choices for $a$ such that $\deg(a)<n-m$, for each of which there are $b^{n(|u|-1)}$ solutions $\tilde{\bsq}_u$ to $\rtr_n(\bsk_u)\cdot \tilde{\bsq}_u\equiv a \pmod p$. Thus we have
  \begin{align*}
    \frac{1}{b^{n|u|}}\sum_{\substack{\tilde{\bsq}_u\in G_{b,n}^{|u|}\\ \rtr_n(\bsk_u)\cdot \tilde{\bsq}_u\equiv a \pmod p\\ \deg(a)<n-m}}1 = \begin{cases}
     1 & \text{if}\quad  p|\rtr_n(k_j) \quad \text{for all}\; j\in u, \\
     \frac{1}{b^m} & \text{otherwise} .
    \end{cases}
  \end{align*}
For $k\in \nat$, suppose $k$ is expressed in the form $lb^n+k'$ such that $l\in \nat_0$ and $0\le k'<b^n$. If $k'=0$, we have $\rtr_n(k)=0$ and thus $p\mid \rtr_n(k)$. Otherwise if $k'>0$, $p\nmid \rtr_n(k)$. Using these results, we obtain
  \begin{align*}
    \bar{B}_{\alpha,\bsgamma,\lambda} \le & \sum_{\emptyset \ne u\subseteq \{1,\ldots,s\}}\gamma_u^{\lambda}C_{\alpha,b}^{\lambda|u|}\left[ \frac{1}{b^m}\sum_{\bsk_u\in \Ecal_b^{|u|}}b^{-2\lambda\mu_{\alpha}(\lfloor\bsk_u/b\rfloor)}+\sum_{\substack{\bsk_u\in \Ecal_b^{|u|}\\ p\mid \rtr_n(k_j), \forall j\in u}}b^{-2\lambda\mu_{\alpha}(\lfloor\bsk_u/b\rfloor)}\right] \\
    = & \sum_{\emptyset \ne u\subseteq \{1,\ldots,s\}}\gamma_u^{\lambda}C_{\alpha,b}^{\lambda|u|}\left[ \frac{1}{b^m}\left(\sum_{k\in \Ecal_b}b^{-2\lambda\mu_{\alpha}(\lfloor k/b\rfloor)}\right)^{|u|}+\left(\sum_{\substack{k\in \Ecal_b\\ b^n\mid k}}b^{-2\lambda\mu_{\alpha}(\lfloor k/b\rfloor)}\right)^{|u|}\right] \\
    \le & \sum_{\emptyset \ne u\subseteq \{1,\ldots,s\}}\gamma_u^{\lambda}C_{\alpha,b}^{\lambda|u|}\left[ \frac{A_{\alpha,b,\lambda,1}^{|u|}}{b^m}+\left(\frac{A_{\alpha,b,\lambda,2}}{b^{4\lambda n}}\right)^{|u|}\right] \\
    \le & \frac{1}{b^{\min(m,4\lambda n)}}\sum_{\emptyset \ne u\subseteq \{1,\ldots,s\}}\gamma_u^{\lambda}C_{\alpha,b}^{\lambda|u|}\left( A_{\alpha,b,\lambda,1}^{|u|}+A_{\alpha,b,\lambda,2}^{|u|}\right), 
  \end{align*}
where we use Lemma \ref{lem:sum-of-digit2} in the second inequality. From (\ref{eq:averaging}), this bound on $\bar{B}_{\alpha,\bsgamma,\lambda}$ is also a bound on $B_{\alpha,\bsgamma}^{\lambda}(\bsq,p)$. Hence the result follows.
\end{proof}

\begin{remark}\label{remark:existence}
If $n\ge \alpha m/2$, we always have $\min(m/\lambda, 4n)=m/\lambda$ for any $1/(2\alpha)<\lambda\le 1$, and thus obtain a bound on $\hat{e}^2(P_{b^m,s}(\bsq,p),\Kcal_{\alpha,\bsgamma})$,
  \begin{align*}
    B_{\alpha,\bsgamma}(\bsq,p) \le \frac{1}{b^{m/\lambda}}\left[\sum_{\emptyset \ne u\subseteq \{1,\ldots,s\}}\gamma_u^{\lambda}C_{\alpha,b}^{\lambda|u|}\left( A_{\alpha,b,\lambda,1}^{|u|}+A_{\alpha,b,\lambda,2}^{|u|}\right)\right]^{1/\lambda} .
  \end{align*}
This compares favorably with the bound on the mean square worst-case error of higher order polynomial lattice rules whose quadrature points are randomly digitally shifted but not folded using the $b$-TT, since $n\ge \alpha m$ is required to achieve the same convergence rate, see \cite[Theorem~4.4]{DP07}. As we cannot achieve the convergence rate of the mean square worst-case error of order $b^{-2\alpha m}$ in $\Hcal_{\alpha,\bsgamma}$ \cite{S63}, our result is almost optimal.
\end{remark}
\noindent
The following corollary of Theorem \ref{theorem:existence} gives sufficient conditions on the weights under which the bound on the information complexity $N(\varepsilon,s)$ does not depend on the dimension, or does depend only polynomially on the dimension.

The initial error in $\Hcal_{\alpha,\bsgamma}$ is given by
  \begin{align*}
     \hat{e}^2(P_{0,s},\Kcal_{\alpha,\bsgamma}) = & \int_{[0,1]^s}\int_{[0,1]^{2s}}\Kcal_{\alpha,\bsgamma}(\bsx,\bsy)\rd \bsx \rd \bsy \rd \bssigma \\
     = & \hat{\Kcal}_{\alpha,\bsgamma}(\bszero,\bszero) = \gamma_{\emptyset} .
  \end{align*}
From Remark \ref{remark:existence}, when $n\ge \alpha m/2$, we have 
  \begin{align*}
     \hat{e}^2(P_{b^m,s}(\bsq,p),\Kcal_{\alpha,\bsgamma})\le \frac{1}{b^{m/\lambda}}\left[\sum_{\emptyset \ne u\subseteq \{1,\ldots,s\}}\gamma_u^{\lambda}C_{\alpha,b}^{\lambda|u|}\left( A_{\alpha,b,\lambda,1}^{|u|}+A_{\alpha,b,\lambda,2}^{|u|}\right)\right]^{1/\lambda} ,
  \end{align*}
for any $1/(2\alpha)<\lambda \le 1$. By considering the inequality $\hat{e}(P_{b^m,s}(\bsq,p),\Kcal_{\alpha,\bsgamma})\le \varepsilon \hat{e}(P_{0,s},\Kcal_{\alpha,\bsgamma})$, the information complexity $N(\varepsilon,s)$ is bounded by
  \begin{align*}
     N(\varepsilon,s) \le \inf_{m\in \nat}\left\{ b^m : \exists \lambda\in \left( \frac{1}{2\alpha},1\right], \frac{1}{b^{m/\lambda}}\left[\sum_{\emptyset \ne u\subseteq \{1,\ldots,s\}}\gamma_u^{\lambda}C_{\alpha,b}^{\lambda|u|}\left( A_{\alpha,b,\lambda,1}^{|u|}+A_{\alpha,b,\lambda,2}^{|u|}\right)\right]^{1/\lambda} \le \varepsilon^{2}\gamma_{\emptyset} \right\} .
  \end{align*}
Since the proof of the following corollary is almost the same as that of \cite[Theorem~5.2]{DP07}, we omit it.

\begin{corollary}\label{cor:tractability}
We define
  \begin{align*}
    G_{\lambda,a}:=\limsup_{s\to \infty}\left[\frac{1}{s^a}\sum_{\emptyset \ne u\subseteq \{1,\ldots,s\}}\gamma_u^{\lambda}C_{\alpha,b}^{\lambda|u|}\left( A_{\alpha,b,\lambda,1}^{|u|}+A_{\alpha,b,\lambda,2}^{|u|}\right)\right] ,
  \end{align*}
for $a\ge 0$ and $1/(2\alpha)<\lambda \le 1$.
\begin{enumerate}
\item Assume $G_{\lambda,0}<\infty$ for some $1/(2\alpha)<\lambda\le 1$. Then $N(\varepsilon,s)$ is bounded independently of the dimension.
\item Assume $G_{\lambda,a}<\infty$ for some $1/(2\alpha)<\lambda\le 1$ and $a>0$. Then the bound on $N(\varepsilon,s)$ depends polynomially on the dimension.
\end{enumerate}
\end{corollary}


\end{document}